\documentclass[reqno,11pt]{amsart}
\usepackage{amsmath,amssymb,latexsym,mathrsfs}
\usepackage{mathtools}
\usepackage[dvipsnames]{xcolor}
\usepackage{graphicx}
\usepackage{enumitem}
\usepackage[colorlinks=true,urlcolor=blue,
citecolor=red,linkcolor=blue,linktocpage,pdfpagelabels,
bookmarksnumbered,bookmarksopen]{hyperref}
\usepackage[english]{babel}
\usepackage[left=2.6cm,right=2.6cm,top=2.9cm,bottom=2.9cm]{geometry}
\usepackage[hyperpageref]{backref}
\usepackage{comment}
\usepackage[normalem]{ulem}
\usepackage{pdfcomment}
\usepackage[textsize=small]{todonotes}

\numberwithin{equation}{section}

\newtheorem{theorem}{Theorem}[section]
\newtheorem{lemma}[theorem]{Lemma}
\newtheorem{corollary}[theorem]{Corollary}
\newtheorem{proposition}[theorem]{Proposition}
\theoremstyle{definition}

\theoremstyle{remark}
\newtheorem{remark}[theorem]{Remark}

\newtheorem{theoremletter}{Theorem}
\newtheorem{propositionletter}[theoremletter]{Proposition}

\newcommand{\ud}{\,\mathrm{d}}
\newcommand{\loc}{\mathrm{loc}}

\DeclareMathOperator{\supp}{supp}
\DeclareMathOperator{\dist}{dist}

\newcommand{\R}{\mathbb{R}}

\newcommand{\Ss}{\mathbb{S}}
\newcommand{\eps}{\varepsilon}

\title[Yamabe stability optimizers]{Existence of Yamabe stability 
optimizers\thanks{This work was partially supported by Fundação de Amparo à Pesquisa 
do Estado de São Paulo (FAPESP, S\~ao Paulo Research Foundation) \#2020/07566-3, 
\#2021/15139-0, and \#2023/15567-8, Conselho Nacional de Pesquisa Científica 
(CNPq, National Council for Scientific and Technological Development) \#409764/2023-0, 
\#443594/2023-6, \#441922/2023-6, and \#306014/2025-4,
Deutsche Forschungsgemeinschaft (DFG, German Research Foundation) \#555837013 and \#561401741, and
Hong Kong General Research Fund (HKGRF) \#14309824.}}

\author[J.\,H.\,Andrade]{Jo\~ao Henrique Andrade}
\address{Department of Mathematics, University of S\~ao Paulo, 05508-090, S\~ao Paulo-SP, Brazil}
\email{andradejh@ime.usp.br}

\author[T.\,K\"onig]{Tobias K\"onig}
\address{Institut f\"ur Mathematik, Goethe-Universit\"at Frankfurt, Robert-Mayer-Str. 10, 60325 Frankfurt am Main, Germany}
\email{koenig@mathematik.uni-frankfurt.de}

\author[J.\,Ratzkin]{Jesse Ratzkin}
\address{Department of Mathematics, Universit\"{a}t W\"{u}rzburg, 97070, W\"{u}rzburg-BA, Germany}
\email{jesse.ratzkin@mathematik.uni-wuerzburg.de}

\author[J.\,Wei]{Juncheng Wei}
\address{Department of Mathematics, Chinese University of Hong Kong, Shatin, NT, Hong Kong}
\email{wei@math.cuhk.edu.hk}

\date{}

\begin{document}

\begin{abstract}
We prove the existence of stability optimizers for the Yamabe inequality on closed Riemannian 
manifolds of dimension at least three with positive Yamabe invariant that satisfy two threshold 
conditions. Remarkably, the compactness threshold 
we uncover is different from the special case of the round sphere treated previously in \cite{arXiv:2211.14185}.
More precisely, it is given by sequences blowing up in one instead of two bubbles, reflecting the compactness of 
Yamabe minimizers in the non-spherical case.
Using the classical asymptotic analysis of Aubin--Schoen test functions, we prove that the stability constant 
is strictly below the one-bubble threshold in dimension at least six and when the manifold is not locally 
conformally flat. In the complementary case, namely in dimensions three through five or when the manifold 
is locally conformally flat, we find a new positive-mass-type condition which is sufficient for the strict 
inequality.
\end{abstract}

\maketitle

\section{Introduction and main result}\label{sec:intro}

The Yamabe problem, posed by Yamabe \cite{Yam} and resolved through successive 
contributions of Trudinger \cite{Tru}, Aubin \cite{Aub}, and Schoen \cite{Sch1} (along 
with the work of many others), 
asks whether each conformal class on a closed Riemannian manifold contains a metric 
of constant scalar curvature. The solution proceeds by minimizing the Yamabe quotient, 
whose infimum defines the Yamabe constant of the conformal class. 
On 
the round sphere $\mathbb S^n$, this minimization is equivalent to the sharp Sobolev 
inequality on Euclidean space, whose extremals were classified independently by 
Aubin \cite{Aub} and Talenti \cite{Tal}, and form a noncompact family parametrized 
by dilations and translations. The reader will find a unified account of the 
resolution of the Yamabe problem in the wonderful survey article by Lee and 
Parker \cite{LP}.

Given a sharp inequality, one naturally asks about its quantitative 
stability. In other words, how well does the deficit between the two sides of the 
inequality control the distance between a given function and the set of optimizers? 
Bianchi and Egnell \cite{BE} obtained the first such estimate for the Sobolev 
inequality on Euclidean space. Their 
result has recently
been refined by several authors; we refer to 
Dolbeault {\it et al.} \cite{DEFFL} and the references therein. In \cite{ENS}, 
Engelstein, Neumayer and Spolaor prove a Bianchi-Egnell-type stability inequality for the 
Yamabe functional on general manifolds (see \eqref{ENS_ineq} below). A fundamental and 
wide open question asks for the exact value of the best constant in the stability 
inequality and the regime which attains equality. The results of \cite{BE, ENS, DEFFL} 
leave open the simpler question whether the best stability constant is itself attained, 
that is, whether there exists a 
function that saturates the stability inequality. Previously, the second 
author \cite{arXiv:2211.14185} proved there exists a stability optimizer 
for the Euclidean Sobolev inequality, and in the present paper we 
investigate the setting of a compact Riemannian manifold. The case of 
general manifolds presents at least two additional conceptual difficulties, namely the possible 
degeneracy of the stability inequality \eqref{ENS_ineq}, and the fact that the Yamabe optimizers 
are generally not explicit.

\medskip
\noindent{\bf Setup and definitions.}
After this short synopsis, let us be more precise.
We let $(M,g)$ be a smooth closed $n$-dimensional Riemannian manifold with $n \geqslant 3$ and 
scalar curvature $R_g$. We consider the quadratic form
\[
\mathcal{E}[u] := \int_M u L_g u \ud V_g = \int_M \left(|\nabla u|^2 + c_n R_g u^2\right) \ud V_g,
\]
where
\[
c_n = \frac{n-2}{4(n-1)} \quad {\rm and} \quad L_g = -\Delta_g + c_n R_g.
\]
The associated \emph{Yamabe quotient functional} is
\[
\mathcal{Q}(u) := \frac{\mathcal{E}[u]}{\|u\|_{L^{p}(M)}^2},
\quad {\rm with} \quad p = 2^*:= \frac{2n}{n-2}.
\]
Recall that $2^*=\tfrac{2n}{n-2}$ is the critical exponent for the Sobolev embedding, which 
was one complicating factor in the original resolution of the Yamabe problem.
By the solution of the Yamabe problem, the \emph{best Yamabe constant}
\[
Y(M,[g]) := \inf_{\substack{u \in \mathcal{C}^\infty(M)\\ u \geqslant 0}}\mathcal{Q}(u)
\]
always admits a minimizer. 
Observe that, by the maximum principle, one can replace the 
inequality $u \geqslant 0$ in the definition of the Yamabe invariant by the strict 
inequality $u>0$. 

On the round sphere $(\Ss^n, g_\circ)$, the Yamabe constant equals
\[
Y(\Ss^n,[g_\circ]) = \frac{n(n-2)}{4}\,\omega_n^{2/n},
\]
where $\omega_n = |\Ss^n|$ is the volume of the unit $n$-sphere. We denote by
\begin{equation}\label{eq:optimizer-set}
\mathscr{O} := \left\{ u \in \mathcal{C}^\infty(M) : \mathcal{Q}(u) = Y(M,[g])\; {\rm and} \;  u \geqslant 0 \right\}
\end{equation}
the nonempty set of Yamabe optimizers.

Since $Y(M,[g]) < Y(\Ss^n,[g_\circ])$ whenever $(M,g)$ is not conformally equivalent 
to the round sphere $(\Ss^n, g_\circ)$, the set
\begin{equation}
\label{O1 compact}
\mathscr{O}_1 := \{ h \in \mathscr{O} : \|h\|_{L^p(M)} = 1 \}
\end{equation}
is compact in $\mathcal{C}^2(M)$. 
This fact is well known, 
see, {\it e.g.}, \cite[Section 4]{LP}. Since the compactness of $\mathscr O_1$ is 
crucial to our arguments in this paper, we give a detailed proof in 
Lemma \ref{lemma compactness} below.

In what follows, we shall always assume that we are in the positive case, namely
\[
Y(M,[g]) > 0,
\]
so that 
\[\|u\| := \mathcal E[u]^{1/2}
\]
induces a norm on $H^1(M)$ that 
is equivalent to the standard norm. Throughout, 
we write 
\[ Y_\circ := Y(\Ss^n,[g_\circ]) \qquad \text{ and } \qquad Y_M := Y(M,[g])
\]
for brevity.
We are interested in the quantitative stability of the Yamabe inequality 
\begin{equation} \label{yamabe_ineq}
\mathcal{E}[u] \geqslant Y(M,[g]) \|u\|_{L^p(M)}^2 \quad {\rm for \ all} \quad u\in H^1(M).
\end{equation}
In the case that $(M,g)$ is conformally equivalent to $(\Ss^n, g_\circ)$, 
by conformal invariance \eqref{yamabe_ineq} reduces to the classical 
Sobolev inequality in Euclidean space, 
\begin{equation}\label{eucl_sobolev} 
\int_{\R^n} |\nabla u|^2 \ud x \geqslant Y(\Ss^n,[g_\circ]) 
\left(\int_{\R^n}|u|^p\ud x\right)^{2/p} \quad {\rm for \ all} \quad u \in \dot{H}^1(\R^n).
\end{equation} 
One can transfer between \eqref{yamabe_ineq} 
and \eqref{eucl_sobolev} when $(M,g) \simeq (\Ss^n, g_\circ)$ using stereographic 
projection. The 
inequality \eqref{eucl_sobolev}
is saturated precisely by scalar multiples of the Aubin--Talenti bubbles 
$U_{\mu,x_0}\in \mathcal{C}^\infty(\R^n)$ defined as 
\[
U_{\mu,x_0}(y) = 
[n(n-2)]^\frac{n-2}{4} \left(\frac{\mu}{\mu^2 + |y - x_0|^2}\right)^{\frac{n-2}{2}}
\quad \text{for some} \quad \mu > 0 \quad \text{and} \quad x_0 \in \R^n.
\]
Here, we observe that the normalizing constant $\alpha_n := [n(n-2)]^{(n-2)/4}>0$ is chosen such that 
\[
-\Delta U_{\mu, x_0} 
= U_{\mu, x_0}^{p-1} \quad {\rm in} \quad \R^n.
\]
Furthermore, Bianchi and Egnell \cite{BE}
proved that the deficit controls the $\dot{H}^1$-distance to the
family of bubbles, namely
\begin{equation*}
    \|\nabla u\|^2_{L^2(\R^n)} - Y(\Ss^n,[g_\circ]) \|u\|_{L^p(\R^n)}^2 \gtrsim 
    \inf_{\substack{\mu > 0,\, x_0 \in \R^n, \\ {c \neq 0}}} \|\nabla(u - 
    c U_{\mu,x_0})\|_{L^2(\R^n)}^2 \quad {\rm for \ all} \quad u \in \dot{H}^1(\R^n).
\end{equation*}

On closed manifolds, the analog of this result was established by Engelstein, Neumayer 
and Spolaor {\cite{ENS}}, who proved the following quantitative stability estimate for the Yamabe inequality.
Here and in the following, we denote 
\begin{equation}
    \label{dist definition}
    \dist(u, \mathscr{O}) := \inf_{h \in \mathscr{O}} \mathcal{E}[u-h]^{1/2} . 
\end{equation} 

\begin{theoremletter}[{\cite{ENS}}]
\label{theorem ENS}
Let $(M,g)$ be a closed $n$-dimensional Riemannian manifold with $n \geqslant 3$ and $Y(M,[g]) > 0$. Then
there exist $\gamma = \gamma(M,g) \geqslant 2$ and $c > 0$ such that
\begin{equation} \label{ENS_ineq} 
\mathcal{E}[u]^{(\gamma-2)/2} \left( \mathcal{E}[u] - Y(M,[g]) \|u\|_{L^p(M)}^2 \right) 
\geqslant c\, \dist(u,\mathscr{O})^\gamma
\quad \text{for all} \quad 0 \leqslant u \in H^1(M). 
\end{equation} 
\end{theoremletter}

\begin{remark}
    In fact, Engelstein, Neumayer and Spolaor prove \eqref{ENS_ineq} for all closed Riemannian 
    manifolds, not just those with positive scalar curvature. However, in the case 
    that $Y(M,[g]) \leqslant 0$ one must measure the distance to the set of 
    optimizers using the either the $L^p$ norm or the standard $H^1$ norm on $M$ instead of \eqref{dist definition}.  
    However, for the level of precision needed 
    in this paper, we essentially cannot afford to lose constants coming from norm equivalence 
    in the computation of stability constants, and we often rely on the favorable algebraic 
    structure given by the definition \eqref{dist definition} of the distance. Hence, our arguments here do 
    not directly carry over to the case $Y(M, [g]) \leqslant 0$, which we leave to future work. 
\end{remark}

One can find examples of closed Riemannian manifolds such that one must take 
$\gamma>2$ in \eqref{ENS_ineq} (cf. \cite{F}), but Engelstein, Neumayer, and Spolaor show that generically 
one can take $\gamma=2$, as in the classical case of the Sobolev inequality. 
Let us denote by $c_{\gamma}(M,g)$ the largest admissible constant in 
Theorem~\ref{theorem ENS}, equivalently
\begin{equation}\label{eq:ENS-def}
c_{\gamma}(M,g) = \inf_{\substack{u \in H^1(M) \setminus \mathscr{O}\\ u\geqslant 0}} \mathcal{S}_\gamma(u),
\end{equation}
where $\mathcal{S}_\gamma: H^1(M) \setminus \mathscr{O} \to \R$ is the \emph{stability quotient}
\begin{equation} \label{defn_stability} 
\mathcal{S}_\gamma(u) := \frac{\mathcal{E}[u]^{(\gamma-2)/2} \left( \mathcal{E}[u] - Y(M,[g]) \|u\|_{L^p(M)}^2 \right)}{\dist(u,\mathscr{O})^\gamma}.
\end{equation} 

A natural question left open by Theorem~\ref{theorem ENS} is whether $c_{\gamma}(M,g)$ 
is \emph{attained}, that is, whether there exists $u \in H^1(M) \setminus \mathscr{O}$ 
satisfying $\mathcal{S}_\gamma(u) = c_{\gamma}(M,g)$. In the case $M = \Ss^n$ 
the second author \cite{arXiv:2211.14185} proved that the infimum is attained, settling a question left 
open since \cite{BE}. His approach identifies two concentration-compactness thresholds:
a \emph{local threshold} arising from sequences converging to an optimizer and a \emph{bubbling 
threshold} arising from concentrating sequences, and shows that their strict separation 
from $c_{\gamma}(\Ss^n, g_\circ)$ implies compactness of minimizing sequences.

In the present paper, we adapt this two-threshold strategy to the general manifold 
setting. As already mentioned above, in doing so we face two conceptual 
challenges which are not present in the Euclidean case, namely the possible degeneracy 
of the stability inequality \eqref{ENS_ineq} (when $\gamma > 2$) and the fact that the 
minimizers are not explicit. To deal with the latter, a crucial use will be the compactness 
of $\mathcal O_1$ proved in Lemma \ref{lemma compactness}.

To include the degenerate setting, 
we need to analyze the convexity properties of the stability quotient $\mathcal S_\gamma$ in 
a setting that allows for general $\gamma \geqslant 2$: the main ingredient here is Lemma \ref{lemma:F-no-local-min}. 
We prove attainability of $c_{\gamma}(M,g)$ under explicit geometric conditions. 

\medskip
\noindent{\bf Main results.}
Our first main result concerns the attainability of the best stability constant $c_{\gamma}(M,g)$. 
The key observation is that two threshold constants govern the attainability, namely
\begin{equation}\label{eq:stability_constant_loc}
c_{\gamma}^{\mathrm{loc}}(M,g):= \inf \left\{ \liminf_{k \to \infty} 
\mathcal{S}_\gamma(u_k) : (u_k)_{k \in \mathbb{N}} \subset H^1(M),\, u_k \to h 
\text{ for some } h \in \mathscr{O} \right\}
\end{equation}
and
\begin{equation}\label{eq:stability_constant_bubble}
c_{\gamma}^{\rm bub}(M,g) := \inf \left\{ \liminf_{k \to \infty} 
\mathcal{S}_\gamma(u_k) : (u_k)_{k \in \mathbb{N}} \subset H^1(M),\, \|u_k\|_{L^p(M)} = 1,
\, u_k \rightharpoonup 0 \right\}.
\end{equation}

The local constant defined in \eqref{eq:stability_constant_loc} captures the 
behavior of minimizing sequences that converge to an optimizer, 
while the constant given in \eqref{eq:stability_constant_bubble} captures the 
behavior of sequences that concentrate and form bubbles. As we show in 
Corollary~\ref{cor:bubble-formula} below,
\[
c_{\gamma}^{\rm bub}(M,g) = 1 - \frac{Y(M,[g])}{Y(\Ss^n,[g_\circ])} \quad {\rm for \ all} \quad \gamma \geqslant 2.
\]

\begin{theorem}
\label{theorem existence stab opt gen M}
Let $(M,g)$ be a closed $n$-dimensional Riemannian manifold with $n \geqslant 3$ such 
that $0 < Y(M,[g]) < Y(\Ss^n,[g_\circ])$ and let $\gamma \geqslant 2$ be the stability
exponent. If 
\begin{equation}
\label{strict ineq thm}
c_{\gamma}(M,g) < c_{\gamma}^{\mathrm{loc}}(M,g) \quad {\rm and} 
\quad c_{\gamma}(M,g) < c_{\gamma}^{\rm bub}(M,g)
\end{equation}
then every minimizing sequence $(u_k)_{k \in \mathbb{N}}$ for $c_{\gamma}(M,g)$ 
converges strongly in $H^1(M)$ up to a subsequence. In particular, in this case 
there exists an optimizer $u \in H^1(M) \setminus \mathscr{O}$ for $c_{\gamma}(M,g)$.
\end{theorem}

Before continuing, we would like to make some remarks. First, observe that the strict 
inequality $Y(M, [g]) < Y(\Ss^n, [g_\circ])$ holds
unless $(M, g)$ is conformally equivalent to the round sphere. Thus, combining 
Theorem~\ref{theorem existence stab opt gen M} with the results of \cite{arXiv:2211.14185}, we 
have a complete characterization of when a stability optimizer exists in the 
positive scalar curvature setting. 

Second, we compare the compactness thresholds from \eqref{strict ineq thm}
with those from \cite{arXiv:2211.14185} concerning the case $M = \Ss^n$. Surprisingly, the 
compactness thresholds are qualitatively different: instead
of $c_{\gamma}^{\rm bub}(M,g)$, in \cite{arXiv:2211.14185} the threshold $2 - 2^{2/p}$ appears, 
which is attained by test functions made of two bubbles, not one. Our condition \eqref{strict ineq thm} gives 
a unified picture that also clarifies the role of $2 - 2^{2/p}$ in \cite{arXiv:2211.14185}.
In both cases, 
one needs to exclude that a normalized minimizing sequence for the stability constant, in addition to its 
weak limit, develops \emph{a bubble that converges weakly to zero}. The fact that on $M = \mathbb S^n$ the 
threshold $2 - 2^{2/p}$ arises from two bubbles in total reflects its particular conformal invariance 
property, which leads to a family of  optimizers $\mathscr{O}_1$ which exhibits itself bubbling behavior. 
On the other hand, for $M$ with $Y(M,[g]) < Y(\Ss^n,[g_\circ])$ the family $\mathscr{O}_1$ is compact by 
Lemma \ref{lemma compactness} and no bubbling within it can occur. 

A similar observation is made in the recent preprint \cite{CGK2025}, which studies the stability of the optimizers for the Hardy-Sobolev 
inequality 
\[
\int_{\R^n} |\nabla u|^2 \ud y - \beta \int_{\R^n} \frac{|u|^2}{|x|^2} \ud y \gtrsim_{\beta,n} 
\left(\int_{\R^n} |u|^{2^*} \ud y\right)^{2/2^*} \quad {\rm for \ all} \quad u\in \dot{H}^1(\R^n).
\]
This is an intermediate problem
in terms of the discussion about bubbling from the last paragraph because it has dilation symmetry, but 
not translation symmetry. In \cite{CGK2025} the authors introduce a constant which they call a 'hidden 
compactness threshold'. This constant is defined as the limit of Aubin-Talenti bubbles wandering off to 
infinity and is therefore in some sense analogous to our constant $c_\gamma^{\rm bub}(M,g)$, {\it i.e.}, it 
captures the symmetries which may cause bubbling but are not inherent in the variational problem 
under study. 

An asymptotic analysis, very similar to that done by Aubin and Schoen in the solution of the 
Yamabe problem, shows that one can always approach the bubble threshold from below. This is 
the content of our second main result.

\begin{theorem}
\label{theorem c < c-bubble}
Let $(M,g)$ be a closed $n$-dimensional Riemannian manifold with $n \geqslant 6$ such 
that $0 < Y(M,[g]) < Y(\Ss^n,[g_\circ])$. If $(M,g)$ is not locally
conformally flat, then
\[
c_{\gamma}(M,g) < c_{\gamma}^{\rm bub}(M,g).
\]
\end{theorem}

In the complementary regime, when $(M,g)$ is locally conformally flat 
or $3 \leqslant n \leqslant 5$, we use Schoen's Green function modified test 
functions \cite{Sch1}. Let $G_{x_0}$ denote the Green function of $L_g$ with 
pole at $x_0 \in M$, normalized so that $L_g G_{x_0} = \delta_{x_0}$, with expansion
\begin{equation}
\label{Green function expansion}
G_{x_0}(x) = \frac{1}{(n-2)\omega_{n-1}}\, r^{2-n} + \mathfrak{m}_{x_0} + \mathcal{O}(r)
\end{equation}
in conformal normal coordinates centered at $x_0$, where $\omega_{n-1} = |\Ss^{n-1}|$ 
and $\mathfrak{m}_{x_0} \in \R$ is the \emph{mass} of the Green function at $x_0$.

We say that the \emph{mass dominance condition} holds at a point $x_0 \in M$ if
the Green function mass $\mathfrak{m}_{x_0}$ satisfies
\begin{equation}
\label{mass dominance}\tag{$\mathcal{M}_{x_0}$}
\beta_n\, \mathfrak{m}_{x_0} > \frac{\gamma}{2}\,(Y(\Ss^n,[g_\circ]) - Y(M,[g]))\, \mathsf{m}_\infty(x_0),
\end{equation}
where $\beta_n > 0$ is an explicit dimensional constant (see \eqref{eq:beta-def}) and
the rescaled stability mass coefficient $\mathsf{m}_\infty(x_0) \geqslant 0$
is defined by
\[\mathsf{m}_\infty(x_0) = \frac{\alpha_n^2\,(n-2)^2\,\omega_{n-1}^2}{|\Ss^n|^{(n-2)/n}}\,
\sup_{h \in \mathscr{O}_1} \left(\int_M h^{p-1}\, G_{x_0} \ud V_g\right)^2
\]
(see Proposition~\ref{prop:moment-sharp}).
We say that $(M,g)$ satisfies the \emph{mass dominance condition} when there exists
$x_0\in M$ such that \eqref{mass dominance} holds.

With this terminology, we can state the next main result.

\begin{theorem}
\label{theorem Schoen}
Let $(M,g)$ be a closed $n$-dimensional Riemannian manifold with $n \geqslant 3$ 
such that $0 < Y(M,[g]) < Y(\Ss^n,[g_\circ])$. If either $n 
\in \{3,4,5\}$ or $(M,g)$ is locally conformally flat and satisfies the mass 
dominance condition \eqref{mass dominance}, then
\[
c_{\gamma}(M,g) < c_{\gamma}^{\rm bub}(M,g).
\]
\end{theorem}

By the positive mass theorem of Schoen and Yau \cite{SY1, SY_PMT2,witten}, one 
has $\mathfrak{m}_{x_0} > 0$ 
for every $x_0 \in M$ whenever $(M,g)$ is not conformally equivalent to $(\Ss^n, g_\circ)$. 
Please see the recent manuscripts \cite{pmt_dim19, brendle_pmt} for proofs 
of the positive mass theorem in any dimension $n \geqslant 3$. 
We refer the reader to \cite{SY1, SY_PMT2, witten, pmt_dim19, brendle_pmt} for the history of the positive mass theorem 
and its various formulations. We emphasize that in our setting the positivity of $\mathfrak{m}_{x_0}$ is used only to verify 
the quantitative mass dominance condition \eqref{mass dominance}, and is not needed for the abstract existence result 
(Theorem~\ref{theorem existence stab opt gen M}).

Condition \eqref{mass dominance} thus reduces to a quantitative comparison between the 
Green function mass and the stability mass coefficient
$\mathsf{m}_\infty(x_0)$. We stress that
this is a new feature of the stability problem: in the Yamabe problem, the
sign $\mathfrak{m}_{x_0} > 0$ alone suffices to conclude $Y(M,[g]) < Y(\Ss^n,[g_\circ])$,
whereas for the stability quotient $\mathcal{S}_\gamma$ the positive stability mass $\mathsf{m}_\infty(x_0)$ creates
an additional constraint.

In the non-locally-conformally-flat case with $n \geqslant 6$, Theorem~\ref{theorem c < c-bubble}
shows that the bubble threshold condition of Theorem~\ref{theorem existence stab opt gen M} is
automatically satisfied, whence it can be dropped.

\begin{corollary}
\label{corollary existence stab opt gen M}
Let $(M,g)$ be a closed $n$-dimensional Riemannian manifold with $n\geqslant 3$ such 
that $0 < Y(M,[g]) < Y(\Ss^n,[g_\circ])$. Suppose that either
\begin{itemize}
\item[(i)] $n \geqslant 6$ and $(M,g)$ is not locally conformally flat, or
\item[(ii)] $n \in \{3,4,5\}$ or $(M,g)$ is locally conformally flat, and the mass dominance condition \eqref{mass dominance} holds.
\end{itemize}
If the strict inequality below holds
\begin{equation}
\label{strict ineq cor}
c_{\gamma}(M,g) < c_{\gamma}^{\mathrm{loc}}(M,g),
\end{equation}
then every minimizing sequence $(u_k)_{k \in \mathbb{N}}$ for $c_{\gamma}(M,g)$ 
converges strongly in $H^1(M)$ up to a subsequence. In particular,
there exists an optimizer $u \in H^1(M) \setminus \mathscr{O}$ for $c_{\gamma}(M,g)$.
\end{corollary}

All three results, Theorems~\ref{theorem existence stab opt gen M},~\ref{theorem c < c-bubble}, 
and~\ref{theorem Schoen}, hold for each Riemannian manifold with a positive Yamabe 
invariant, regardless of $\gamma \geqslant 2$. The test function expansions 
in \S\ref{sec:proof-thm2}--\S\ref{sec:proof-schoen} are independent of $\gamma$, while the 
compactness argument in Theorem~\ref{theorem existence stab opt gen M} is treated in full 
generality in Claim~2 of the proof in \S\ref{sec:proof-thm1}. For Theorem~\ref{theorem Schoen}, 
the mass dominance condition \eqref{mass dominance} depends on $\gamma$ through the factor 
$\gamma/2$, which arises from the expansion of $\dist(u_\mu, \mathscr{O})^\gamma$ in the 
denominator of $\mathcal{S}_\gamma$ (cf. \eqref{eq:S-expansion-schoen}). Since the 
left-hand side of \eqref{mass dominance} is independent of $\gamma$ while the right-hand 
side grows linearly in $\gamma$, the condition becomes more restrictive as $\gamma$ increases. 
For $\gamma = 2$, the factor $\gamma/2$ equals $1$ and \eqref{mass dominance} reduces to
\[
\beta_n\, \mathfrak{m}_{x_0}- (Y(\Ss^n,[g_\circ]) - Y(M,[g]))\,\mathsf{m}_\infty(x_0)> 0.
\]

\medskip
\noindent{\bf Strategy of the proof.}
The proof of Theorem~\ref{theorem existence stab opt gen M} follows a concentration-compactness-type 
argument adapted from \cite{arXiv:2211.14185}. 
Because of the setting 
on a compact manifold and the compactness of minimizers $\mathscr O_1$, the concentration 
alternative for a minimizing sequence can be treated in a very elementary and efficient  
fashion solely in terms of weak convergence $u_k \rightharpoonup 0$ (compare the definition 
of $c_{\gamma}^{\rm bub}(M,g)$).
Similarly to \cite{arXiv:2211.14185}, it is then important to exclude two unfavorable 
alternatives:  either the sequence converges to an optimizer $h \in \mathscr{O}$, in which case 
\[
\liminf_{k \to \infty} \mathcal{S}_\gamma(u_k) \geqslant c_{\gamma}^{\mathrm{loc}}(M,g),
\]
or $u_k \rightharpoonup 0$ ({\it i.e.}\ all of $u$ concentrates), in which case
\[
\liminf_{k \to \infty} \mathcal{S}_\gamma(u_k) \geqslant c_{\gamma}^{\rm bub}(M,g).
\]
The two strict inequalities in \eqref{strict ineq thm} exclude both scenarios, forcing the
sequence to converge strongly to $u \in H^1(M) \setminus \mathscr O$. The reduction to the model 
scenarios of $c_\gamma^{\rm loc}$ and $c_\gamma^{\rm bub}$ uses
a one-variable auxiliary function that encodes the asymptotic behavior
of $\mathcal{S}_\gamma$ along concentrating sequences; the compactness argument then relies on a
calculus lemma showing that this function has no local minima (Lemma~\ref{lemma:F-no-local-min}). This 
approach is similar to \cite{arXiv:2211.14185} but more complicated because it needs to take into account the 
more complicated structure of $\mathcal S_\gamma$ when $\gamma > 2$. 

The verification of the bubble threshold condition, carried out
in \S\ref{sec:proof-thm2}--\S\ref{sec:proof-schoen}, introduces a difficulty absent in the
classical Yamabe problem. The \emph{stability mass} $\mathsf{m}$ defined in \eqref{eq:momentum}
appears in the Taylor expansion of the stability quotient as a second-order mass term that competes against the energy deficit in
the numerator. In the Yamabe problem, the sign of the energy correction alone determines whether the
Aubin inequality holds strictly. For the stability quotient, however, the energy correction
and the stability mass have the same sign but opposite
effects: a negative energy correction
drives the stability quotient of the Aubin test functions below the bubble threshold, while the positive stability mass pushes it
back up. For Aubin's test functions on non-LCF manifolds with $n \geqslant 6$, the Weyl curvature
correction at order $\mu^4$ dominates the stability mass at order $\mu^{n-2}$,
and the bubble threshold is satisfied unconditionally. For Schoen's test functions in the LCF case,
both terms are at order $\mu^{n-2}$, and the comparison gives rise to the mass dominance
condition \eqref{mass dominance}.

\medskip
\noindent{\bf Examples.}
In the forthcoming companion paper \cite{AKRW-cylinder}, we verify the hypotheses of 
Theorem \ref{theorem existence stab opt gen M} for the explicit family of product 
cylinders 
\[
(M_\tau,g_\tau):= (\mathbb S^1(\tau) \times \mathbb S^{n-1},g_{\mathbb S^1(\tau)}\otimes g_{\mathbb S^{n-1}}).
\] 
On these manifolds, we compute the local
stability constant $c_{\gamma}^{\mathrm{loc}}(M_\tau, g_\tau)$ explicitly via the
spectral decomposition on the cylinder.
Indeed, for 
\(\tau \leqslant \tau^*:= \tfrac{1}{\sqrt{n-2}}\)
the minimizers are uniquely the constant 
functions \cite{Schoen1989}, which allows for computing all the relevant quantities to 
Theorems \ref{theorem existence stab opt gen M} and \ref{theorem Schoen} (that is, 
$c_\gamma$, $c_\gamma^{\rm loc}$, $c_\gamma^{\rm bub}$, $\mathfrak m(x_0)$ 
and $\mathsf m_\infty(x_0)$)  explicitly and determine sufficient conditions for the 
existence of stability optimizers. 

The family $M_\tau$ is of particular interest for 
two reasons. The first is that when $\tau = \tau^*$ the manifold $(M_\tau, g_\tau)$ 
exhibits degenerate stability with $\gamma = 4$ (see \cite{ENS, F}). The second reason
is that cylinders provide possible blow-up models for scalar curvature. It is well known
that there exist blow-up sequences of solutions to the Yamabe equation \cite{Khuri_Marques_Schoen, DHR}, and that the
blow-up occurs in isolated points. More precisely, there exists a Riemannian manifold
$(M,g)$ of dimension $n$ and a sequence
of conformal factors $(u_k)_{k\in\mathbb{N}}\subset H^1(M)$ such that the conformal metric $g_k = u_k^{{4}/{(n-2)}}
g$ has constant scalar curvature equal to $n(n-1)$ and
\(\limsup_{k \in \mathbb{N}} \| u_k \|_{L^\infty(M)} =\infty.\)
If we let $p_k=\max_{x\in M}u_k(x)$ for all $k\in\mathbb{N}$ denote a maximum point
and let $p_\infty=\limsup_{k\to\infty}p_k$ be an accumulation
point of the sequence $(p_k)_{k\in\mathbb{N}}$.
For each $k\in\mathbb{N}$, we can rescale the solution $u_k\in
\mathcal{C}^2(B_{R_k}^*(p_k))$ about $p_k$, where $B_{R_k}^*(p_k):=B_{R_k}(p_k)
\setminus\{p_k\}$ is the punctured ball centered at $p_k$ with radius
$R_k=\mathrm{o}_k(1)$. In the limit, we obtain $u_k\to u_\infty$ as $k\to\infty$,
where $u_\infty \in \mathcal{C}^2(\R^n \setminus \{ 0 \})$ is an entire
solution on the punctured space, which is conformally equivalent to the cylinder
$\R \times \Ss^{n-1}$. Under certain circumstances, solutions on the cylinder
$\R \times \Ss^{n-1}$ are periodic and thus descend to the product $\Ss^1(\tau)
\times \Ss^{n-1}$ for some $\tau>0$. Heuristically, this explanation is the
same reason ancient solutions of Ricci flow are important in classifying possible
blow-up behavior of Ricci flow in general.

\medskip
\noindent{\bf Organization of the paper.}
In \S\ref{sec:prelim} we collect the distance formula and basic properties of the functional $\mathsf{m}$, 
and we compute the bubble threshold $c_{\gamma}^{\rm bub}(M,g)$. In \S\ref{sec:proof-thm1} we prove 
the abstract existence theorem. In \S\ref{sec:proof-thm2} we verify the bubble threshold condition via 
Aubin's test functions for $n \geqslant 6$, non-LCF manifolds. In \S\ref{sec:proof-schoen} we carry out 
the Schoen test function construction and prove the mass dominance criterion. A list of notation is collected at the end of the paper.

\section{Preliminaries}\label{sec:prelim}

\subsection{Compactness of the set of $L^p$-normalized optimizers} 
The following lemma guarantees the compactness of the set $\mathscr O_1$ (defined in \eqref{O1 compact}) 
of $L^p$-normalized Yamabe optimizers whenever $Y(M,[g]) < Y(\Ss^n,[g_\circ])$. Since it is an important 
tool throughout our arguments, we provide a proof following \cite[Section 4]{LP}. 

\begin{lemma}
    \label{lemma compactness}
    Let $(M,g)$ be a smooth closed $n$-dimensional Riemannian manifold with $n \geqslant 3$ 
    such that $Y(M,[g]) < Y(\Ss^n,[g_\circ])$. Then $\mathscr O \subset \mathcal C^\infty(M)$. 
    Moreover, the set $\mathscr O_1$ is compact in $\mathcal C^2(M)$. 
  \end{lemma}

\begin{proof}
Each $u \in \mathscr O_1$ satisfies, in the weak sense, the Euler-Lagrange 
equation $L_g u = Y_M u^{p-1}$, where the value $Y_M$ is determined from testing 
the equation with $u$ and using $\|u\|_{L^p(M)} = 1$ and $\mathcal E(u) = Y_M$. Thus 
the regularity theorem stated in \cite[Theorem 4.1]{LP} implies $u \in \mathcal C^\infty(M)$.

For the compactness statement, we claim there exist $q > p$ and $C > 0$ such 
that $\|u\|_{L^q(M)} \leqslant C$ for every $u \in \mathscr O_1$. Accepting this 
claim for the moment, it follows again from \cite[Theorem 4.1]{LP} that  there 
exist $\alpha, C > 0$ such that $\|u\|_{\mathcal C^{2, \alpha}(M)} \leqslant C$ for 
every $u \in \mathscr O_1$. The compactness statement then follows from the Arzel\`a--Ascoli theorem.

To prove the claim, we slightly adapt the argument in \cite[proof of Proposition 4.4]{LP}. 
For some $\delta > 0$ to be determined, we test the equation $L_g u = Y_M u^{p-1}$ with 
$u^{1 + 2 \delta}$, which is an admissible test function because $u \in \mathcal C^\infty(M)$. 
This yields
\begin{equation}\label{eq:test_integral}
\int_M \left( (1 + 2 \delta)  u^{2\delta} |\nabla u|^2  + 
c_n R_g u^{2 + 2\delta} \right) \ud V_g  = Y_M \int_M u^{p + 2 \delta} \ud V_g.
\end{equation}
Setting $w := u^{1 + \delta}$, we have $\nabla w = (1 + \delta) u^\delta \nabla u$, 
and so \eqref{eq:test_integral} becomes 
\begin{align} \label{compactness1}
\frac{1 + 2 \delta}{( 1+ \delta)^2} \|\nabla w\|_{L^2(M)}^2 &=
c_n \int_M R_g u^{2 + 2\delta} \ud V_g  \\ \nonumber
&=  Y_M \int_M w^2 u^{p-2} \ud V_g - c_n \int_M R_g w^2 \ud V_g \\ \nonumber
&\leqslant  Y_M \|w\|_{L^p(M)}^2 + C \|w\|_{L^2(M)}^2,
\end{align} 
by H\"older's inequality and $\|u\|_{L^p(M)} = 1$.

By Aubin's Sobolev inequality \cite[Theorem 2.3]{LP}, for every $\eps > 0$ there is $C_\eps > 0$ such that 
\begin{equation}
    \label{aubin}
     \|w\|_{L^p(M)}^2 \leqslant \frac{1 + \eps}{Y_\circ} \|\nabla w\|_{L^2(M)}^2 + C_\eps \|w\|_{L^2(M)}^2 .
\end{equation}
Combining \eqref{aubin} and \eqref{compactness1} we get 
\begin{equation} \label{compactness2} 
\|w\|_{L^p(M)}^2 \leqslant  (1 + \eps) \frac{(1+\delta)^2}{1 + 2 \delta} 
\frac{Y_M}{Y_\circ} \|w\|_{L^p(M)}^2 + C_{\delta, \eps} \|w\|_{L^2(M)}^2.   
\end{equation}
Since $Y_M < Y_\circ$, we can pick $\delta, \eps > 0$ sufficiently small 
such that 
\[(1 + \eps) \frac{(1+\delta)^2}{1 + 2 \delta} \frac{Y_M}{Y_\circ}< 1,\] 
which in turn allows us to rearrange \eqref{compactness2} as 
\[ \|w\|_{L^p(M)}^2 \leqslant C \|w\|_{L^2(M)}^2 \]
for some $C > 0$ depending on $\delta, \eps$, but independent of $u \in \mathscr O_1$. 
Rewriting this last equation in terms of $u$, we find 
\[ \|u\|_{L^{p(1 + \delta)}(M)}^{2(1 + \delta)} 
\leqslant C \|u\|_{L^{2(1 + \delta)}(M)}^2 \leqslant C 
\|u\|_{L^p(M)}^2 = C \quad {\rm for \ all} \quad u \in \mathscr O_1. \]
The application of H\"older's inequality in the last estimate is justified 
whenever $2(1+ \delta) \leqslant p$, which we can always achieve by decreasing 
$\delta > 0$ if necessary. This proves the claim, with $q = p(1 +\delta) > p$.
\end{proof}

\begin{remark}
    The parameter $\eps$ in the proof above is unnecessary if we replace Aubin's inequality \eqref{aubin}
    with its strengthened endpoint version with $\eps = 0$ and $C_0 < \infty$, which is due to Hebey and 
    Vaugon \cite{HeVa}. Since this is not essential for our proof, we choose to work directly with 
    the weaker and more classical inequality \eqref{aubin}. The same remark applies to the 
    proof of Proposition \ref{proposition Sobolev for bubbles}, where we use inequality \eqref{aubin} again. 
\end{remark}

By the maximum principle, when $Y(M,[g]) \leqslant 0$, there is (up to scaling) only one 
constant scalar curvature metric in each conformal class. This provides a much simpler 
proof of Lemma~\ref{lemma compactness} in the case of nonpositive 
scalar curvature. 

\subsection{Distance estimates} 
Recall that $Y(M, [g])> 0$ allows to define 
\[\dist(u, \mathscr{O})^2 = \inf_{h \in \mathscr O} \mathcal E[u - h] 
= \inf_{h \in \mathcal O_1, \lambda \in (0, \infty)} \mathcal E[u - \lambda h]. \] 

The proof of the following lemma is nearly identical to that of \cite[Lemma~3.3]{DEFFL}. The 
only slight difference is that in our geometric case the minimization is only over $\lambda 
\in (0,\infty)$ instead of $\lambda \in \R$, which however turns out to pose no problems because 
the minimum $\lambda^*$ is automatically positive.

\begin{lemma}
\label{lemma DEFFL distance}
Let $(M,g)$ be a closed $n$-dimensional Riemannian manifold with $n \geqslant 3$ such that $0 < Y(M,[g]) < Y(\Ss^n,[g_\circ])$.
For every $u \in H^1(M)$, one has
\[
\dist(u, \mathscr{O})^2 = \mathcal E[u] - Y(M,[g]) \mathsf{m}(u),
\]
where the \emph{stability mass}
\begin{equation}\label{eq:momentum}
    \mathsf{m}(u) := \sup_{h \in \mathscr{O}_1} \left( \int_M h^{p-1} u \ud V_g \right)^2
\end{equation}
measures the squared $L^p$-projection of $u$ onto the optimizer manifold $\mathscr{O}_1$.
Moreover, the supremum in \eqref{eq:momentum} is attained and 
\[
\mathsf{m}(u) \leqslant \|u\|_{L^p(M)}^2 \quad {\rm for \ all} \quad u \in H^1(M).
\]
\end{lemma}

\begin{proof}
Initially, since $\mathscr{O}_1$ is compact and consists of positive solutions of $L_g h 
= Y_M h^{p-1}$ with $\|h\|_{L^p(M)} = 1$, the supremum defining $\mathsf{m}(u)$ is attained. Next, 
for $h \in \mathscr{O}_1$ and $\lambda \in (0,\infty)$, 
integration by parts and the equation $L_g h = Y_M h^{p-1}$ yield
\[
\mathcal{E}[u - \lambda h] = \mathcal{E}[u] - 2\lambda Y_M \int_M h^{p-1} u \ud V_g + \lambda^2 Y_M.
\]

This is minimal at $\lambda^* = \int_M h^{p-1} u \ud V_g \in (0, \infty)$, and hence
\[
\inf_{\lambda \in \R} \mathcal{E}[u - \lambda h] = \mathcal{E}[u] - Y_M \left(\int_M h^{p-1} u \ud V_g\right)^2.
\]
Finally, taking the infimum over $h \in \mathscr{O}_1$ on both sides, one obtains $\dist(u, \mathscr{O})^2 = \mathcal{E}[u] - Y_M\, \mathsf{m}(u)$. For the upper bound, since $\|h\|_{L^p(M)} = 1$ for every $h \in \mathscr{O}_1$, a classical H\"older's inequality yields
\[
\left|\int_M h^{p-1} u \ud V_g\right| \leqslant \|h^{p-1}\|_{L^{p/(p-1)}(M)}\, \|u\|_{L^p(M)} = \|h\|_{L^p(M)}^{p-1}\, \|u\|_{L^p(M)} = \|u\|_{L^p(M)}.
\]
Squaring and taking the infimum over $h \in \mathscr{O}_1$, one obtains $\mathsf{m}(u) \leqslant \|u\|_{L^p(M)}^2$.
\end{proof}

\begin{remark}[Terminology]\label{rmk:stability-mass}
We call $\mathsf{m}:H^1(M)\setminus \mathcal{O}\to \mathbb{R}$ the \emph{stability mass} functional because it plays, for the stability
quotient $\mathcal{S}_\gamma$, a role analogous to that of the Green function
mass $\mathfrak{m}_{x_0}$ in the Yamabe problem. Indeed, the Taylor expansion of
$\mathcal{S}_\gamma$ along concentrating test functions (cf.\ \eqref{eq:S-expansion-final}) reads
\[
\mathcal{S}_\gamma(u_\mu) = c_\gamma^{\rm bub}
  + \frac{Y_M}{Y_\circ^2}\, a_\mu
  + \frac{\gamma\, Y_M(Y_\circ - Y_M)}{2\, Y_\circ^2}\, m_\mu
  + \mathrm{o}(a_\mu) + \mathrm{o}(m_\mu),
\]
where $a_\mu = \mathcal{E}[u_\mu] - Y_\circ$ is the energy correction controlled by
$\mathfrak{m}_{x_0}$ and $m_\mu = \mathsf{m}(u_\mu)$ is a second-order correction
that is absent from the Yamabe quotient. For Schoen's test functions both
contributions are at order $\mu^{n-2}$, and the mass dominance
condition \eqref{mass dominance} is precisely the requirement that the Green
function mass dominates the stability mass.
\end{remark}

Thanks to the compactness of $\mathscr{O}_1$, the functional $\mathsf{m}$ vanishes along weakly convergent sequences.

\begin{lemma}
\label{lemma m to 0 weak conv}
Let $(M,g)$ be a closed $n$-dimensional Riemannian manifold with $n \geqslant 3$ 
such that $0 < Y(M,[g]) < Y(\Ss^n,[g_\circ])$.
Suppose that $(u_k)_{k \in \mathbb{N}} \subset H^1(M)$ satisfies $\|u_k\|_{L^p(M)} = 1$ and $u_k \rightharpoonup 0$ in $L^p(M)$. Then, it holds $\mathsf{m}(u_k) = \mathrm{o}_k(1)$.
\end{lemma}

\begin{proof}
Let $(h_k)_{k \in \mathbb{N}} \subset \mathscr{O}_1$ be such that 
$\mathsf{m}(u_k) = (\int_{M} h_k^{p-1} u_k \ud V_g)^2$. By compactness of $\mathscr{O}_1$, 
there is a subsequence $h_{k_j} \to h$ uniformly. Since $h^{p-1} \in L^{p'}(M)$ and 
$u_{k_j} \rightharpoonup 0$ in $L^p(M)$, it follows that $\int_M h^{p-1} u_{k_j} 
\ud V_g \to 0$. By the triangle inequality,
\[
\mathsf{m}(u_{k_j}) = \left(\int_{M} h_{k_j}^{p-1} u_{k_j} \ud V_g\right)^2 \to 0 
\quad {\rm as} \quad k \to \infty.
\]
The same argument, applied to every subsequence, yields $\mathsf{m}(u_k) = \mathrm{o}_k(1)$ for 
the full sequence.
\end{proof}

We use the following continuity property of the stability mass $\mathsf{m}$ in Claim~2 of the proof of Theorem~\ref{theorem existence stab opt gen M}. It follows from the compactness of $\mathscr{O}_1$ and does not require a Br\'ezis--Lieb-type splitting for $\mathsf{m}$.

\begin{lemma}
\label{lemma m continuity}
Let $(M,g)$ be a closed $n$-dimensional Riemannian manifold with $n \geqslant 3$ 
such that $0 < Y(M,[g]) < Y(\Ss^n,[g_\circ])$.
If $(u_k)_{k \in \mathbb{N}} \subset H^1(M)$ with $u_k = u + v_k$
and $v_k \rightharpoonup 0$ in $H^1(M)$, then
\[
\mathsf{m}(u_k) = \mathsf{m}(u) + \mathrm{o}_k(1).
\]
\end{lemma}

\begin{proof}
For the upper bound, let $h_* \in \mathscr{O}_1$ achieve $\mathsf{m}(u)$. Since $h_*^{p-1} 
\in L^{p'}(M)$ and $v_k \rightharpoonup 0$ in $L^p(M)$, one has 
\[
\int_M h_*^{p-1} v_k \ud V_g =\mathrm{o}_k(1), 
\]
which in turn yields
\[
\mathsf{m}(u_k) \leqslant \left(\int_M h_*^{p-1} u_k \ud V_g\right)^2 = \left(\int_M h_*^{p-1} u 
\ud V_g + \mathrm{o}_k(1)\right)^2 = \mathsf{m}(u) + \mathrm{o}_k(1).
\]
For the lower bound, let $(h_k)_{k \in \mathbb{N}} \subset \mathscr{O}_1$ achieve 
$\mathsf{m}(u_k)$. By compactness of $\mathscr{O}_1$, there exists $h \in \mathscr{O}_1$ 
such that after passing to a subsequence $h_k \to h$ uniformly. Then, since $\mathsf{m}(u) \leqslant (\int_M h^{p-1} u \ud V_g)^2$, it holds
\[
\mathsf{m}(u_k) = \left(\int_M h_k^{p-1} u_k \ud V_g\right)^2 = \left(\int_M h^{p-1} u 
\ud V_g + \mathrm{o}_k(1)\right)^2 \geqslant \mathsf{m}(u) + \mathrm{o}_k(1).
\]
Finally, after taking further 
subsequences of this convergent subsequence, we see that  
the convergence holds for the full sequence.
\end{proof}

\subsection{Bubbling analysis}
The following proposition is a key ingredient in the proof of Theorem~\ref{theorem existence stab opt gen M}. It shows that one can improve the Yamabe constant from $Y_M$ to $Y_\circ$ along sequences that converge weakly to zero.

\begin{proposition}
\label{proposition Sobolev for bubbles}
Let $(M,g)$ be a closed $n$-dimensional Riemannian manifold with $n \geqslant 3$ 
such that $0 < Y(M,[g]) < Y(\Ss^n,[g_\circ])$.
If $(v_k)_{k \in \mathbb{N}} \subset H^1(M)$ with $v_k \rightharpoonup 0$ in $H^1(M)$ and $\|v_k\|_{L^p(M)} = 1$, then
\[
(Y(\Ss^n,[g_\circ]) + \mathrm{o}_k(1))  \leqslant \mathcal E[v_k] \quad {\rm as} \quad k \to \infty.
\]
\end{proposition}

\begin{proof}
Recall that we abbreviate $Y_\circ = Y(\Ss^n,[g_\circ])$. 
    By Aubin's Sobolev inequality \cite{Aubin1976JDG} (see also \cite[Theorem 2.3]{LP}), for every $\eps > 0$ there is $C_\eps > 0$ such that 
    \[ \|\nabla u\|_{L^2(M)}^2 \leqslant (Y_\circ^{-1} + \eps) \|u\|_{L^p(M)}^2 + C_\eps \|u\|_{L^2(M)}^2. \]
    For $(v_k)_{k \in \mathbb{N}} \subset H^1(M)$ with $v_k \rightharpoonup 0$ in $H^1(M)$ and $\|v_k\|_{L^p(M)} = 1$, the inequality below holds 
    \[
    \left|\int_M R_g v_k^2 \, \ud V_g\right| \leqslant \|R_g\|_{L^\infty(M)}\, \|v_k\|_{L^2(M)}^2,
    \]
    which, by a simple rearrangement, yields
    \[ 
    \mathcal E[v_k] \geqslant (Y_\circ - \eps) - C_\eps \|v_k\|_{L^2(M)}^2 \quad {\rm for \ all} \quad k\in\mathbb{N}. 
    \]
    Since the embedding of $H^1(M)$ into $L^2(M)$ is compact, we have $\|v_k\|_{L^2(M)} \to 0$. 
    Therefore, for each $k\in\mathbb{N}$, we can choose $\eps(k)=\mathrm{o}_k(1)$ tending to zero sufficiently slowly so that $C_{\eps(k)} \|v_k\|_{L^2(M)}^2 =\mathrm{o}_k(1)$ as $k \to \infty$. For this choice of $=\mathrm{o}_k(1)$, the statement from the proposition follows.
\end{proof}

From Proposition~\ref{proposition Sobolev for bubbles}, we deduce the explicit expression 
for $c_{\gamma}^{\rm bub}(M,g)$.

\begin{corollary}
\label{cor:bubble-formula}
Let $(M,g)$ be a closed $n$-dimensional Riemannian manifold with $n \geqslant 3$ 
such that $0 < Y(M,[g]) < Y(\Ss^n,[g_\circ])$ and let $\gamma \geqslant 2$. Then, it holds
\begin{equation}\label{eq:bubble_formula}
    c_{\gamma}^{\rm bub}(M,g) = 1 - \frac{Y(M,[g])}{Y(\Ss^n,[g_\circ])}.
\end{equation}
\end{corollary}

\begin{proof}
Let $(u_k)_{k \in \mathbb{N}}$ be any sequence with $\|u_k\|_{L^p(M)} = 1$ and $u_k \rightharpoonup 0$ in $H^1(M)$. In addition, since $\dist(u_k, \mathscr{O})^2 \leqslant \mathcal E[u_k]$, one has
\begin{align*}
\mathcal{S}_\gamma(u_k) = \frac{\mathcal E[u_k]^{(\gamma-2)/2} \left( \mathcal{E}[u_k] - Y_M \right)}{\dist(u_k,\mathscr{O})^\gamma}
&\geqslant \frac{\mathcal E[u_k]^{(\gamma-2)/2} \left( \mathcal{E}[u_k] - Y_M \right)}{\mathcal E[u_k]^{\gamma/2}} \\
&= 1 - \frac{Y_M}{\mathcal E[u_k]}\\
&\geqslant 1 - \frac{Y_M}{Y_\circ} + \mathrm{o}_k(1),
\end{align*}
where the last step uses Proposition~\ref{proposition Sobolev for bubbles}. Since $(u_k)_{k \in \mathbb{N}}$ was arbitrary, one has 
\[
c_\gamma^{\rm bub} \geqslant 1 - \frac{Y_M}{Y_\circ}.
\]

For the reverse inequality, we place a single concentrating Aubin--Talenti bubble about 
a fixed point on $M$. The resulting normalized sequence $(u_k)_{k \in \mathbb{N}}$ satisfies $\|u_k\|_{L^p(M)} = 1$ and
\begin{equation}
\label{E dec prop}
\mathcal E[u_k] = Y_\circ + \mathrm{o}_k(1);
\end{equation}
see \cite[Lemma~3.4]{LP} and the more detailed computation in Proposition~\ref{prop:aubin-energy} below. By Lemmas~\ref{lemma DEFFL distance} and~\ref{lemma m to 0 weak conv}, it follows
\begin{equation}
\label{d dec prop}
\dist(u_k, \mathscr{O})^2 = \mathcal E[u_k] + \mathrm{o}_k(1).
\end{equation}
Combining \eqref{E dec prop} and \eqref{d dec prop}, we obtain
\[
\mathcal{S}_\gamma(u_k) = \frac{Y_\circ^{(\gamma-2)/2}(Y_\circ - Y_M)}{Y_\circ^{\gamma/2}} + \mathrm{o}_k(1) = 1 - \frac{Y_M}{Y_\circ} + \mathrm{o}_k(1),
\]
which proves the asymptotics \eqref{eq:bubble_formula}.
\end{proof}

\section{Proof of Theorem~\ref{theorem existence stab opt gen M}}\label{sec:proof-thm1}
This section is devoted to proving the existence of optimizers for the stability quotient $\mathcal S_\gamma$. We begin with a technical calculus lemma that is used in Claim~3 of the proof.

\begin{lemma}
\label{lemma:F-no-local-min}
Let $\gamma \geqslant 2$, $p> 2$, and $a \geqslant b \geqslant c > 0$ with $b < 1$. Let $\Psi:(0,\infty)\to \R$ be the positive function defined as
\[
\Psi(X) = \frac{(a + X)^{(\gamma-2)/2}\left(a + X - b(1 + X^{p/2})^{2/p}\right)}{(a + X - c)^{\gamma/2}}.
\]
Then, $\Psi$ does not admit any local minimum on $(0, \infty)$.
\end{lemma}

\begin{proof}
We note that $\Psi$ is smooth on $(0, +\infty)$. We show that for any point $Y \in (0, +\infty)$ with $\Psi'(Y) = 0$, one has $\Psi''(Y) < 0$, from which the assertion follows immediately.

To make calculations efficient and transparent, we abbreviate
\[
\psi(X) := a + X, \quad h(X) := a + X - b\left(1 + X^{p/2}\right)^{2/p} = \psi(X) - b\left(1 + X^{p/2}\right)^{2/p}
\]
as well as
\[
\alpha := \frac{\gamma - 2}{2} \quad {\rm and} \quad \beta := \frac{\gamma}{2}.
\]
We note that $\psi$ and $h$ are both positive on $(0, +\infty)$ as a consequence of $a \geqslant b > 0$. Moreover, $\psi' \equiv 1$, and $h$ is increasing and concave. Indeed, setting $\phi(X) = (1 + X^{p/2})^{2/p}$, one verifies that $\phi'$ is increasing with $0 < \phi'(X) < 1$, so that $h' = 1 - b\phi'$ is decreasing with $0 < h' < 1$, where we used that $b < 1$.

With these shorthands, $\Psi$ reads
\[
\Psi(X) = \frac{\psi(X)^\alpha\, h(X)}{(\psi(X) - c)^\beta} = A(X)\, h(X),
\]
where
\[
A(X) := \frac{\psi(X)^\alpha}{(\psi(X) - c)^\beta}.
\]
We compute, for any $X \in (0, +\infty)$,
\[
\Psi'(X) = A'(X)\, h(X) + A(X)\, h'(X) = A(X)\left(h'(X) + \varsigma(X)\, h(X)\right),
\]
where we introduced the logarithmic derivative
\[
\varsigma(X) := \frac{A'(X)}{A(X)} = (\log A)'(X) = \frac{\alpha}{\psi(X)} - \frac{\beta}{\psi(X) - c}.
\]
Since $A > 0$, every critical point $Y \in (0, +\infty)$ of $\Psi$ satisfies
\begin{equation}
\label{eq:h-q-critical}
h'(Y) + \varsigma(Y)\, h(Y) = 0.
\end{equation}

A direct computation of the second derivative shows that, at any critical point $Y$ satisfying \eqref{eq:h-q-critical},
\[
\Psi''(Y) = A(Y)\left(h''(Y) + \left(\varsigma'(Y) - \varsigma(Y)^2\right) h(Y)\right).
\]
Since $A > 0$, $h > 0$, and $h'' < 0$ by concavity, it suffices to prove that $\varsigma'(X) - \varsigma(X)^2 \leqslant 0$ for every $X \in (0, +\infty)$. A direct computation yields
\begin{align*}
\varsigma' - \varsigma^2 &= -\frac{\alpha}{\psi^2} + \frac{\beta}{(\psi - c)^2} - \frac{\alpha^2}{\psi^2} + \frac{2\alpha\beta}{\psi(\psi-c)} - \frac{\beta^2}{(\psi-c)^2} \\
&= -\frac{\alpha + \alpha^2}{\psi^2} + \frac{\beta - \beta^2}{(\psi-c)^2} + \frac{2\alpha\beta}{\psi(\psi-c)}.
\end{align*}
Notice that the exponents $\alpha = \tfrac{\gamma-2}{2}$ and $\beta = \tfrac{\gamma}{2}$ satisfy the identity
\[
\alpha + \alpha^2 = \alpha\beta = -(\beta - \beta^2) = \frac{\gamma(\gamma-2)}{4}.
\]
Finally, by substituting, one obtains
\begin{align*}
\varsigma' - \varsigma^2 &= \frac{\gamma(\gamma-2)}{4}\left(-\frac{1}{\psi^2} + \frac{2}{\psi(\psi-c)} - \frac{1}{(\psi-c)^2}\right) = -\frac{\gamma(\gamma-2)}{4}\left(\frac{1}{\psi} - \frac{1}{\psi - c}\right)^2 \leqslant 0,
\end{align*}
since $\gamma \geqslant 2$. It follows that $\Psi''(Y) < 0$ at every critical point $Y$, which proves the lemma.
\end{proof}

Now we are ready to prove our first main result in this manuscript

\begin{proof}[Proof of Theorem~\ref{theorem existence stab opt gen M}]
We divide the proof into three claims.
Let $(u_k)_{k \in \mathbb{N}}$ be a minimizing sequence for $c_\gamma$, normalized so that $\|u_k\|_{L^p(M)} = 1$. 

\medskip
\noindent{\bf Claim~1} {\rm (Nontrivial weak limit):} {\it $(u_k)_{k \in \mathbb{N}}$ has a nontrivial weak limit $u$ in $H^1(M)$.}

\smallskip
\noindent{\it Proof.} First, $(u_k)_{k \in \mathbb{N}}$ is bounded in $H^1(M)$, since otherwise $\mathcal{E}[u_k] \to +\infty$ and, using $\dist(u, \mathscr{O})^2 \leqslant \mathcal E[u]$, one would obtain
\[
c_\gamma = \lim_{k \to \infty} \mathcal{S}_\gamma(u_k) \geqslant \lim_{k \to \infty} \frac{\mathcal E[u_k] - Y_M}{\mathcal E[u_k]} = 1,
\]
which contradicts the inequality
\[
c_\gamma\leqslant c_\gamma^{\rm bub} = 1 - \frac{Y_M}{Y_\circ} < 1.
\]

By the Banach--Alaoglu theorem, there exists $u \in H^1(M)$ such that, up to a subsequence, $u_k \rightharpoonup u$ in $H^1(M)$, $u_k \to u$ in $L^2(M)$, and $u_k(x) \to u(x)$ a.e.\ on $M$. If $u \equiv 0$, then $(u_k)_{k \in \mathbb{N}}$ is admissible for $c_\gamma^{\rm bub}$ and we obtain
\[
c_\gamma = \liminf_{k \to \infty} \mathcal{S}_\gamma(u_k) \geqslant c_\gamma^{\rm bub},
\]
contradicting $c_\gamma < c_\gamma^{\rm bub}$. Hence we may write $u_k = u + v_k$ with $u \not\equiv 0$, $v_k \rightharpoonup 0$ in $H^1(M)$, and $v_k \to 0$ in $L^2(M)$. This proves Claim~1.

\medskip
\noindent{\bf Claim~2} {\rm (Strong convergence):} {\it $v_k \to 0$ strongly in $H^1(M)$.}

\smallskip
\noindent{\it Proof.} The key ingredients are the convexity of $t \mapsto (1 + t^{p/2})^{2/p}$ and the strict inequality $c_\gamma < c_\gamma^{\rm bub}$, together with a calculus lemma (Lemma~\ref{lemma:F-no-local-min}) that rules out local minima of the auxiliary function $\Psi$ below.

Assume again that $u_k = u + v_k$ with $\|u\|_{L^p(M)} = 1$ and $v_k \rightharpoonup 0$ in $H^1(M)$. Using the decompositions from above, one has
\begin{equation}\label{eq:F-quotient}
c_\gamma + \mathrm{o}_k(1) = \mathcal{S}_\gamma(u_k) = \frac{(\mathcal{E}[u] + \mathcal{E}[v_k])^{(\gamma-2)/2}\left(\mathcal{E}[u] + \mathcal{E}[v_k] - Y_M(1 + \|v_k\|_{L^p(M)}^p)^{2/p}\right)}{(\mathcal{E}[u] + \mathcal{E}[v_k] - Y_M\mathsf{m}(u))^{\gamma/2}}.
\end{equation}
By Proposition~\ref{proposition Sobolev for bubbles}, one obtains
\[
\|v_k\|_{L^p(M)}^p \leqslant Y_\circ^{-p/2}\,\mathcal{E}[v_k]^{p/2} + \mathrm{o}_k(1).
\]
We set $a,b,c\in (0,\infty)$ such that $a \geqslant b \geqslant c$ as
\[
a := \frac{\mathcal{E}[u]}{Y_\circ}, \quad b := \frac{Y_M}{Y_\circ}, \quad c := \frac{Y_M\,\mathsf{m}(u)}{Y_\circ}, \quad {\rm and} \quad X_k := \frac{\mathcal{E}[v_k]}{Y_\circ}.
\]
Dividing both the numerator and denominator of \eqref{eq:F-quotient} by $Y_\circ^{\gamma/2}$, one obtains
\[
c_\gamma + \mathrm{o}_k(1) \geqslant \Psi(X_k),
\]
where
\begin{equation}\label{eq:F-def}
\Psi(X) := \frac{(a + X)^{(\gamma-2)/2}\left(a + X - b(1 + X^{p/2})^{2/p}\right)}{(a + X - c)^{\gamma/2}}.
\end{equation}

By Lemma~\ref{lemma:F-no-local-min}, $\Psi$ has no local minimum on $(0,+\infty)$, and so its infimum satisfies
\[
\inf_{X \in (0,+\infty)} \Psi(X) = \min\left\{\lim_{X \to 0^+} \Psi(X),\, \lim_{X \to +\infty} \Psi(X)\right\}.
\]
A direct computation yields
\[
\lim_{X \to 0^+} \Psi(X) = \mathcal{S}_\gamma(u) \geqslant c_\gamma \quad {\rm and} \quad \lim_{X \to +\infty} \Psi(X) = c_\gamma^{\rm bub}.
\]
Since $\Psi(X_k) \leqslant c_\gamma + \mathrm{o}_k(1)$ and $c_\gamma < c_\gamma^{\rm bub}$ by hypothesis, the sequence $(X_k)$ cannot tend to $+\infty$. On the one hand, if there exists a subsequence $(X_{k_j})$ such that $X_{k_j} \to X^*$ for some $X^* \in (0,+\infty)$, then by continuity $\Psi(X^*) \leqslant c_\gamma$. On the other hand, one has
\[
\Psi(X^*) \leqslant \Psi(0^+) \quad {\rm and} \quad  \Psi(X^*) < \Psi(+\infty),
\]
which would imply that $\Psi(X^*)=\inf_{X\in (0,+\infty)}\Psi(X)$ attains its infimum over $(0,+\infty)$ at the interior point, producing a local minimum and contradicting Lemma~\ref{lemma:F-no-local-min}. Therefore $X_k \to 0$, whence $v_k \to 0$ in $H^1(M)$, and Claim~2 is proved.

\medskip
\noindent{\bf Claim~3} {\rm (Non-optimality):} {\it $u \notin \mathscr{O}$.}

\smallskip
\noindent{\it Proof.} Since $v_k \to 0$ strongly in $H^1(M)$ by Claim~2, we have $u_k \to u$ in $H^1(M)$. If $u \in \mathscr{O}$, then $(u_k)$ would be an admissible test sequence for $c_\gamma^{\loc}$, giving
\[
c_\gamma = \liminf_{k \to \infty} \mathcal{S}_\gamma(u_k) \geqslant c_\gamma^{\loc},
\]
which contradicts $c_\gamma < c_\gamma^{\loc}$. Therefore $u \notin \mathscr{O}$, and Claim~3 is proved.

\medskip
Combining Claims~1--3, $u$ is an admissible optimizer for $c_\gamma$, which completes the proof.
\end{proof}

\section{Proof of Theorem~\ref{theorem c < c-bubble}}\label{sec:proof-thm2}

In this section, we prove that $c_{\gamma}(M,g) < c_{\gamma}^{\rm bub}(M,g)$ for manifolds of 
dimension $n \geqslant 6$ that are not locally conformally flat. The strategy is to construct a 
concentrating family of test functions $(u_\mu)_{\mu > 0}$ with $\|u_\mu\|_{L^p(M)} = 1$ and 
$u_\mu \rightharpoonup 0$ in $H^1(M)$, and to show that $\mathcal{S}_\gamma(u_\mu) < 
c_{\gamma}^{\rm bub}(M,g)$ for $0 < \mu \ll 1$. Since $c_{\gamma}(M,g) \leqslant 
\mathcal{S}_\gamma(u_\mu)$, this yields the desired strict inequality.

The construction follows the classical approach of Aubin \cite{Aub}, using truncated Aubin--Talenti 
bubbles in conformal normal coordinates (see Lee and Parker \cite{LP}).

\subsection{Aubin's test functions}
We recall the classical test function construction of Aubin \cite{Aub} in conformal normal coordinates, 
which we then use to evaluate the stability quotient. To do so, we introduce some terminology.
Let $x_0 \in M$ be a point where the Weyl tensor does not vanish, $W_g(x_0) \neq 0$. Let 
$\tilde{g} = \Lambda^{4/(n-2)} g$ be a conformal normal metric and let $\{ y^i\}$ be normal 
coordinates with respect to $\tilde{g}$ centered at $x_0$, such that 
$\det(\tilde{g}_{ij})(y) = 1 + \mathcal{O}(|y|^N)$ for some $N >5$. (Such a conformal metric 
always exists by Theorem 5.1 
of \cite{LP}.) Let $\chi \in \mathcal{C}^\infty_c(\R^n)$ be a smooth cutoff function with 
$\chi \equiv 1$ near the origin, and define the family $(U_\mu)_{\mu > 0}$ in $\mathcal{C}^\infty(\R^n)$ by
\begin{equation}
\label{eq:talenti-aubin}
  U_\mu(y) = \alpha_n \left(\frac{\mu}{\mu^2 + |y|^2}\right)^{\frac{n-2}{2}} \quad {\rm with} 
  \quad \alpha_n = [n(n-2)]^{\frac{n-2}{4}},
\end{equation}
the standard Aubin--Talenti bubble normalized so that it solves
\[
-\Delta U_\mu = U_\mu^{p-1} \quad {\rm in} \quad  \R^n.
\]
Next, we set the family $(\hat{u}_\mu)_{\mu>0}$ as
\[
\hat{u}_\mu(x) = \Lambda(x)^{-1} \chi(y) U_\mu(y) \quad {\rm with} \quad x = \exp_{x_0}^{\tilde{g}}(y),
\]
and consider the normalization by $(u_\mu)_{\mu>0}$ with
\begin{equation}\label{eq:bubbles_family}
u_\mu := \frac{\hat{u}_\mu}{\|\hat{u}_\mu\|_{L^p(M)}},
\end{equation}
so that $\|u_\mu\|_{L^p(M)} = 1$ for all $\mu>0$.

The classical computation of Aubin \cite{Aub} (see also the expansions at the 
bottom of page 58 of Lee and Parker \cite{LP}) produces the following expansion of the Yamabe 
quotient. Since we use some of the same techniques to 
estimate the functional $\mathsf{m}$ in \eqref{eq:momentum}, we include a sketch of the
proof in \S\ref{sec:appendix-energy}. 

\begin{propositionletter}[\cite{Aub}, \cite{LP}]
\label{prop:aubin-energy}
Let $(M, g)$ be a closed $n$-dimensional Riemannian manifold with $n \geqslant 6$. If $(M,g)$ is not locally conformally flat, then the conformal energy of the normalized Aubin family $(u_\mu)_{\mu>0}$ defined in \eqref{eq:bubbles_family} satisfies
\[
\mathcal{E}[u_\mu] = Y(\Ss^n,[g_\circ]) + a_\mu,
\]
where
\[
a_\mu =
\begin{cases}
-C_n |W_g(x_0)|^2 \mu^4 + \mathrm{o}(\mu^4), & n \geqslant 7, \\[4pt]
-C_6 |W_g(x_0)|^2 \mu^4 \log(1/\mu) + \mathcal{O}(\mu^4), & n = 6,
\end{cases}
\]
for explicit constants $C_n > 0$ depending only on $n$.
\end{propositionletter}

\begin{remark}
   We stress that the negativity of $a_\mu$ is the key ingredient in the proof of
   Theorem~\ref{theorem c < c-bubble} below. Indeed, the stability
   expansion \eqref{eq:S-expansion-final} involves the coefficient
   $Y_M/Y_\circ^2 > 0$ in front of $a_\mu$, so that $a_\mu < 0$ pushes
   $\mathcal{S}_\gamma(u_\mu)$ below $c_{\gamma}^{\rm bub}(M,g)$. For $n \geqslant 7$
   the term $m_\mu = \mathcal{O}(\mu^{n-2})$ is of strictly higher order than $|a_\mu|
   = \mathcal{O}(\mu^4)$, and for $n = 6$ the factor $\log(1/\mu)$ in $|a_\mu|$ dominates
   $m_\mu = \mathcal{O}(\mu^4)$, so in both cases the energy correction alone determines the sign of
   $\mathcal{S}_\gamma(u_\mu) - c_{\gamma}^{\rm bub}(M,g)$.
 
\end{remark}

\subsection{The stability mass}
We now estimate $\mathsf{m}(u_\mu)$ defined in Lemma~\ref{lemma DEFFL distance}. 
Since $u_\mu$ concentrates at $x_0$ and $\|u_\mu\|_{L^p(M)} = 1$, we have $u_\mu \rightharpoonup 0$ 
in $H^1(M)$, whence $\mathsf{m}(u_\mu) \to 0$ by Lemma~\ref{lemma m to 0 weak conv}. The precise 
rate is as follows.

\begin{lemma}
\label{lemma:moment-rate}
Let $(M, g)$ be a closed $n$-dimensional Riemannian manifold with $n \geqslant 3$. Then $\mathsf{m} (u_\mu)$ 
satisfies 
\[
\mathsf{m}(u_\mu) = \mathcal{O}(\mu^{n-2}) \quad {\rm as} \quad \mu \to 0 , 
\]
where $(u_\mu)_{\mu > 0}$ is the normalized Aubin family defined in \eqref{eq:bubbles_family}.
\end{lemma}

\begin{proof}
Let $h_\mu \in \mathscr{O}_1$ achieve $\mathsf{m}(u_\mu)$. By the compactness of $\mathscr{O}_1$, $h_\mu \leqslant C$ on $M$, uniformly in $\mu$. 

A direct computation in conformal normal coordinates yields
\[
\int_M h_\mu^{p-1} u_\mu \ud V_g \leqslant C \int_M u_\mu \ud V_g. 
\]
The $L^1$ norm of $u_\mu$ on $M$ satisfies $\int_M u_\mu \ud V_g = \mathcal{O}(\mu^{(n-2)/2})$, since in conformal normal coordinates,
\begin{align*}
\int_{B_R(0)} U_\mu(y) \ud y &= \alpha_n \mu^{(n-2)/2 + 2}|\Ss^{n-1}| \int_0^{R/\mu} \frac{s^{n-1}}{(1 + s^2)^{(n-2)/2}} \ud s= \mathcal{O}\!\left(\mu^{(n-2)/2}\right),
\end{align*}
where we used $\int_0^{R/\mu} s^{n-1}(1+s^2)^{-(n-2)/2} \ud s = \mathcal{O}(\mu^{-2})$. It follows that
\[
\mathsf{m}(u_\mu) = \left(\int_M h_\mu^{p-1} u_\mu \ud V_g\right)^2 = \mathcal{O}(\mu^{n-2}).  \qedhere
\]
\end{proof}
For later use, we observe that for $0 < \mu \ll 1$ the $\mathsf{m}$-term satisfies the bound 
\[\mathsf{m}(u_\mu) = \mathcal{O}(\mu^{n-2}) = \mathrm{o}(|a_\mu|) \quad
{\rm as} \quad \mu \searrow 0,\] whenever the dimension $n$ is at least $6$ and the Weyl tensor of $g$ does not 
vanish at $x_0$. 

\subsection{Completion of the proof}

The proof follows from the energy and $\mathsf{m}$-estimates established above.

\begin{proof}[Proof of Theorem~\ref{theorem c < c-bubble}]
We use the test functions $u_\mu\in \mathcal{C}^\infty(M)$ constructed above. We set
\[
E_\mu := \mathcal{E}[u_\mu] = Y_\circ + a_\mu \quad {\rm and} \quad m_\mu := \mathsf{m}(u_\mu),
\]

\medskip
\noindent{\bf Claim~1} {\rm (Stability quotient expansion):} {\it $\mathcal{S}_\gamma(u_\mu) = c_\gamma^{\rm bub} + \frac{Y_M}{Y_\circ^2}\, a_\mu + \frac{\gamma\, Y_M(Y_\circ - Y_M)}{2\, Y_\circ^2}\, m_\mu + \mathrm{o}(a_\mu) + \mathrm{o}(m_\mu)$.}

\smallskip
\noindent{\it Proof.} By the definition of $\mathcal{S}_\gamma$ and the identity $\dist(u_\mu, \mathscr{O})^2 = E_\mu - Y_M\, m_\mu$ from Lemma~\ref{lemma DEFFL distance}, one has
\begin{equation}\label{eq:S-expansion}
\mathcal{S}_\gamma(u_\mu) = \frac{E_\mu^{(\gamma-2)/2}(E_\mu - Y_M)}{(E_\mu - Y_M\, m_\mu)^{\gamma/2}}
= \left(1 - \frac{Y_M}{E_\mu}\right)\left(1 - \frac{Y_M\, m_\mu}{E_\mu}\right)^{-\gamma/2}.
\end{equation}
Since $a_\mu, m_\mu \to 0$, we expand each factor around $E_\mu = Y_\circ$. For the first factor, one has
\[
\frac{1}{E_\mu} = \frac{1}{Y_\circ + a_\mu} = \frac{1}{Y_\circ}\left(1 - \frac{a_\mu}{Y_\circ} + \mathcal{O}(a_\mu^2)\right),
\]
whence
\[
1 - \frac{Y_M}{E_\mu} = 1 - \frac{Y_M}{Y_\circ} + \frac{Y_M}{Y_\circ^2}\, a_\mu + \mathcal{O}(a_\mu^2) = c_\gamma^{\rm bub} + \frac{Y_M}{Y_\circ^2}\, a_\mu + \mathcal{O}(a_\mu^2),
\]
where we used $c_\gamma^{\rm bub} = 1 - \tfrac{Y_M}{Y_\circ}$ from Corollary~\ref{cor:bubble-formula}. For the second factor, since 
\[
\frac{Y_M m_\mu}{E_\mu} = \frac{Y_M m_\mu}{Y_\circ} + \mathcal{O}(a_\mu m_\mu),
\]
the binomial expansion yields
\[
\left(1 - \frac{Y_M\, m_\mu}{E_\mu}\right)^{-\gamma/2} = 1 + \frac{\gamma\, Y_M}{2\, Y_\circ}\, m_\mu + \mathcal{O}(m_\mu^2).
\]
Multiplying these two factors, one obtains
\begin{equation}\label{eq:S-expansion-final}
\begin{aligned}
\mathcal{S}_\gamma(u_\mu) &= c_\gamma^{\rm bub} + \frac{Y_M}{Y_\circ^2}\, a_\mu + \frac{\gamma\, Y_M\, c_\gamma^{\rm bub}}{2\, Y_\circ}\, m_\mu + \mathrm{o}(a_\mu) + \mathrm{o}(m_\mu)\\
&= c_\gamma^{\rm bub} + \frac{Y_M}{Y_\circ^2}\, a_\mu + \gamma\, \frac{Y_M(Y_\circ - Y_M)}{2\, Y_\circ^2}\, m_\mu + \mathrm{o}(a_\mu) + \mathrm{o}(m_\mu), 
\end{aligned}
\end{equation}
where we used that $c_\gamma^{\rm bub} = \tfrac{Y_\circ - Y_M}{Y_\circ}$. This proves Claim~1.

\medskip
\noindent{\bf Claim~2} {\rm (Strict inequality):} {\it $\mathcal{S}_\gamma(u_\mu) < c_\gamma^{\rm bub}$ for $0 < \mu \ll 1$.}

\smallskip
\noindent{\it Proof.} We divide the proof into two cases as follows.

\noindent \textit{Case $n \geqslant 7$.}
By Proposition~\ref{prop:aubin-energy}, one has $a_\mu = -C_n |W_g(x_0)|^2 \mu^4 + \mathrm{o}(\mu^4)$. By Lemma~\ref{lemma:moment-rate}, $m_\mu = \mathcal{O}(\mu^{n-2})$. Since $n - 2 >4$, it follows that $m_\mu = \mathrm{o}(\mu^4) = \mathrm{o}(|a_\mu|)$, whence the energy correction dominates in \eqref{eq:S-expansion-final} and one obtains
\[
\mathcal{S}_\gamma(u_\mu) = c_\gamma^{\rm bub} - \frac{Y_M\, C_n}{Y_\circ^2}\, |W_g(x_0)|^2\, \mu^4 + \mathrm{o}(\mu^4) < c_\gamma^{\rm bub}
\]
for $0 < \mu \ll 1$.

\smallskip
\noindent\textit{Case $n = 6$.}
By Proposition~\ref{prop:aubin-energy}, one has $a_\mu = -C_6 |W_g(x_0)|^2 \mu^4 \log(1/\mu) + \mathcal{O}(\mu^4)$. By Lemma~\ref{lemma:moment-rate}, $m_\mu = \mathcal{O}(\mu^4)$. Since $\mu^4 \log(1/\mu) \gg \mu^4$ as $\mu \to 0$, the energy correction still dominates, whence
\[
\mathcal{S}_\gamma(u_\mu) = c_\gamma^{\rm bub} - \frac{Y_M\, C_6}{Y_\circ^2}\, |W_g(x_0)|^2\, \mu^4 \log(1/\mu) + \mathcal{O}(\mu^4) < c_\gamma^{\rm bub}
\]
for $0 < \mu \ll 1$.
In both cases, $c_\gamma \leqslant \mathcal{S}_\gamma(u_\mu) < c_\gamma^{\rm bub}$, and Claim~2 is proved.

\medskip
Combining Claims~1 and~2 completes the proof.
\end{proof}

\begin{remark}
\label{rmk:lcf-and-low-dim}
The assumption that $(M,g)$ is not locally conformally flat is used to guarantee $W_g(x_0) \neq 0$ at some point, which produces the negative energy correction $a_\mu < 0$. For locally conformally flat manifolds one has $W_g \equiv 0$, so this mechanism is unavailable. For manifolds of dimension $3 \leqslant n \leqslant 5$, regardless of whether they are locally conformally flat, the Aubin energy correction is $a_\mu = \mathcal{O}(\mu^4)$ while the $\mathsf{m}$-correction is $m_\mu = \mathcal{O}(\mu^{n-2})$ with $n - 2 <4$, so the $\mathsf{m}$-term dominates and the Weyl tensor approach fails even when $W_g \neq 0$. In both situations one must use Schoen's Green function modified test functions, for which the energy and $\mathsf{m}$-corrections are of the same order $\mu^{n-2}$, and the sign of the leading coefficient depends on the mass dominance condition \eqref{mass dominance}. In this fashion, Theorems~\ref{theorem c < c-bubble} and~\ref{theorem Schoen} together cover all closed manifolds with $0 < Y_M < Y(\Ss^n, [g_\circ])$. We carry out the Schoen analysis in \S\ref{sec:proof-schoen}.
\end{remark}

\begin{proof}[Proof of Corollary~\ref{corollary existence stab opt gen M}]
Since $(M,g)$ is not locally conformally flat and $n \geqslant 6$, Theorem~\ref{theorem c < c-bubble} yields $c_\gamma < c_\gamma^{\rm bub}$. Together with \eqref{strict ineq cor}, both strict inequalities in \eqref{strict ineq thm} hold, and the conclusion follows from Theorem~\ref{theorem existence stab opt gen M}.
\end{proof}

\section{Proof of Theorem~\ref{theorem Schoen}}\label{sec:proof-schoen}

We construct Schoen's Green function modified test functions and compute the corresponding stability quotient.

\subsection{Schoen's test functions}
Let $x_0 \in M$ be a point at which the mass dominance condition \eqref{mass dominance} holds. 
Since $Y_M > 0$, the Green function $G_{x_0}$ of $L_g$ is strictly positive on 
$M \setminus \{x_0\}$ and admits the expansion \eqref{Green function expansion}. We work in 
conformal normal coordinates $y$ centered at $x_0$ and set
\[
\beta_\mu := \alpha_n\, (n-2)\, \omega_{n-1}\, \mu^{(n-2)/2},
\]
so that 
\[
U_\mu(y) \sim \frac{\beta_\mu}{(n-2)\omega_{n-1}}\, |y|^{2-n} \quad {\rm as} \quad 
\frac{|y|}{\mu} \to +\infty.
\]
Let $r_\mu>0$ satisfy $\mu \ll r_\mu \ll 1$, for instance $r_\mu = \mu^{1/2}$, and let $\eta_\mu$ 
be a smooth cutoff with $\eta_\mu \equiv 1$ on $B_{r_\mu}(x_0)$ and $\eta_\mu \equiv 0$ 
outside $B_{2r_\mu}(x_0)$. We define
\[
\hat{w}_\mu(x) :=
\begin{cases}
U_\mu(y), & |y| \leqslant r_\mu, \\[3pt]
\eta_\mu(y)\, U_\mu(y) + (1 - \eta_\mu(y))\, \beta_\mu\, G_{x_0}(x), & r_\mu
\leqslant |y| \leqslant 2 r_\mu, \\[3pt]
\beta_\mu\, G_{x_0}(x), & |y| \geqslant 2 r_\mu,
\end{cases}
\]
and consider the normalization $(w_\mu)_{\mu>0}$ with
\begin{equation}\label{eq:schoen_family}
 w_\mu := \frac{\hat{w}_\mu}{\|\hat{w}_\mu\|_{L^p(M)}},
\end{equation}
so that $\|w_\mu\|_{L^p(M)} = 1$ for all $\mu > 0$. By construction,
$w_\mu \rightharpoonup 0$ in $H^1(M)$ as $\mu \to 0$.

The following classical estimate is equation (2.8) of Schoen \cite{Sch1} (see 
also \cite{Schoen1989} or the combination of Proposition 7.1 and Lemma 9.7 of \cite{LP}). For 
the sake of completeness, we added a sketch of the proof to Appendix \S\ref{sec:appendix-energy}.

Once again we introduce some notation.
Let $x_0 \in M$ and $G_{x_0}$ be the Green function of $L_g$ with pole at $x_0$ and mass 
$\mathfrak{m}_{x_0}$ as in \eqref{Green function expansion}, and let $(w_\mu)_{\mu>0}$ 
be the normalized Schoen test function defined above.

\begin{propositionletter}[{\cite{Sch1}}, {\cite{Schoen1989}}]
\label{prop:schoen-energy}
Let $(M, g)$ be a closed $n$-dimensional Riemannian manifold with $n \geqslant 3$ 
and $0 < Y(M,[g]) < Y(\Ss^n,[g_\circ])$.
Then, the conformal energy of the normalized Schoen family $(w_\mu)_{\mu>0}$ defined 
in \eqref{eq:schoen_family} satisfies
\[
\mathcal{E}[w_\mu] = Y(\Ss^n,[g_\circ]) + a_\mu,
\]
where
\begin{equation}
\label{eq:schoen-energy}
a_\mu = -\beta_n\, \mathfrak{m}_{x_0}\, \mu^{n-2} + \mathrm{o}(\mu^{n-2})
\end{equation}
with
\begin{equation}
\label{eq:beta-def}
\beta_n :=\frac{2^{n-2}\,\alpha_n^2\,(n-2)^2\,\omega_{n-1}^2}{|\Ss^n|^{\frac{n-2}{n}}} > 0.
\end{equation}
\end{propositionletter}

\begin{remark}\label{rmk:beta-explicit}
Notice that using 
\[\alpha_n = [n(n-2)]^{\frac{n-2}{4}}, \quad \omega_{n-1} = {2\pi^{\frac{n}{2}}}{\Gamma(\tfrac{n}{2})^{-1}}, \quad {\rm and} \quad |\Ss^n| = {2\pi^{\frac{n+1}{2}}}{\Gamma\!\left(\tfrac{n+1}{2}\right)^{-1}},
\]
the constant $\beta_n=\beta(n)$ is an explicit function of $n$, namely
\begin{equation}
\label{eq:beta-gamma}
\beta(n) = 2^{\frac{n^2-n+2}{n}}\,\pi^{\frac{n^2+n+2}{2n}}\,n^{\frac{n-2}{2}}\,(n-2)^{\frac{n+2}{2}}\,\Gamma\!\bigl(\tfrac{n+1}{2}\bigr)^{\frac{n-2}{n}}{\Gamma\!\bigl(\tfrac{n}{2}\bigr)^{-2}}.
\end{equation}
\end{remark}

\subsection{The mass stability functional}

The same argument as in Lemma~\ref{lemma:moment-rate} yields $\mathsf{m}(w_\mu) = \mathcal{O}(\mu^{n-2})$. Indeed, for any $h_* \in \mathscr{O}_1$, one has

\[
\int_M h_*^{p-1}\, w_\mu \ud V_g = h_*(x_0)^{p-1} \int_M w_\mu \ud V_g + \mathcal{O}\!\left(\int_M |x - x_0|\, w_\mu \ud V_g\right),
\]
and the $L^1$ norm of $w_\mu$ satisfies $\int_M w_\mu \ud V_g = \mathcal{O}(\mu^{(n-2)/2})$, since
\[
w_\mu(x) \lesssim \beta_\mu\, G_{x_0}(x) = \mathcal{O}(\mu^{(n-2)/2}) \quad {\rm on} \quad M \setminus B_{r_\mu}(x_0)
\]
and $\int_{B_\mu} U_\mu \ud y = \mathcal{O}(\mu^{n/2})$. In particular, one has
\[
\mathsf{m}(w_\mu) = \mathcal{O}(\mu^{n-2}) \quad {\rm as} \quad \mu \to 0.
\]
The last estimate allows us to define
\begin{equation}
\label{eq:minfty-def}
\mathsf{m}_\infty(x_0) := \limsup_{\mu \to 0}\, \mu^{-(n-2)}\, \mathsf{m}(w_\mu) \geqslant 0.
\end{equation}

\begin{proposition}
\label{prop:moment-sharp}
Let $(M, g)$ be a closed $n$-dimensional Riemannian manifold with $n \geqslant 3$ and $0 < Y(M,[g]) < Y(\Ss^n,[g_\circ])$. Then, it holds
\begin{equation}
\label{eq:minfty-formula}
\mathsf{m}_\infty(x_0) =\lim_{\mu \to 0} \mu^{-(n-2)}\, \mathsf{m}(w_\mu) = \frac{[\alpha_n\, (n-2)\, \omega_{n-1}]^2}{\|U_1\|_{L^p(\R^n)}^2}\, \sup_{h \in \mathscr{O}_1} \left(\int_M h^{p-1}\, G_{x_0} \ud V_g\right)^2,
\end{equation}
where $U_1 \in \mathcal{C}^\infty(\R^n)$ is the unscaled Aubin--Talenti bubble in \eqref{eq:talenti-aubin}.
\end{proposition}

\begin{proof}
Since $w_\mu = \hat{w}_\mu / \|\hat{w}_\mu\|_{L^p(M)}$ and
\[
\|\hat{w}_\mu\|_{L^p(M)}^p = \int_{\R^n} U_\mu^p \ud y + \mathrm{o}(1) = \|U_1\|_{L^p(\R^n)}^p + \mathrm{o}(1),\]
one has $\|\hat{w}_\mu\|_{L^p(M)} \to \|U_1\|_{L^p(\R^n)}$ as $\mu \to 0$. We decompose $M$ into three regions: the core $\Omega_1 := B_{r_\mu}(x_0)$, the transition annulus $\Omega_2 := B_{2r_\mu}(x_0) \setminus B_{r_\mu}(x_0)$, and the outer region $\Omega_3 := M \setminus B_{2r_\mu}(x_0)$. For any $h \in \mathscr{O}_1$, we estimate $\int_{\Omega_j} h^{p-1}\, \hat{w}_\mu \ud V_g$ on each region separately, uniformly in $h \in \mathscr O_1$.

\medskip
\noindent{\bf Claim~1} {\rm (Core integral):} {\it $\int_{\Omega_1} h^{p-1}\, \hat{w}_\mu \ud V_g = \mathcal{O}(\mu^{n/2})$ uniformly in $h \in \mathscr{O}_1$.}

\smallskip
\noindent{\it Proof.} In the core region $B_{r_\mu}(x_0)$, one has $\hat{w}_\mu = U_\mu$. Since $h$ is smooth and positive, it holds $h(x)^{p-1} = h(x_0)^{p-1} + \mathcal{O}(|y|)$ in conformal normal coordinates, whence
\[
\int_{B_{r_\mu}} h^{p-1}\, U_\mu \ud V_g = h(x_0)^{p-1} \int_{B_{r_\mu}} U_\mu \ud y + \mathcal{O}\!\left(\int_{B_{r_\mu}} |y|\, U_\mu \ud y\right).
\]
The rescaling $y = \mu z$ yields
\[
\int_{B_{r_\mu}} U_\mu \ud y = \alpha_n \mu^{(n-2)/2}\, \mu^n\, \mu^{-(n-2)} \int_{B_{r_\mu/\mu}} (1 + |z|^2)^{-(n-2)/2} \ud z = \alpha_n \mu^\frac{n+2}{2}\int_0^{r_\mu/\mu} \frac{|\Ss^{n-1}|\, s^{n-1}}{(1 + s^2)^{(n-2)/2}} \ud s.
\]
Since the integrand decays as $s^{n-1-(n-2)} = s$ for $s \gg 1$, it holds
\[
\int_0^{r_\mu/\mu} s^{n-1}(1+s^2)^{-(n-2)/2} \ud s=\mathcal{O}(r_\mu/\mu)^2 \quad {\rm with} \quad r_\mu = \sqrt{\mu},
\]
one obtains $(r_\mu/\mu)^2 = \mu^{-1}$, and so
\[
\int_{B_{r_\mu}} U_\mu \ud y = \mathcal{O}(\mu^\frac{n+2}{2} \mu^{-1}) = \mathcal{O}(\mu^{n/2}).
\]
This proves Claim~1.

\medskip
\noindent{\bf Claim~2} {\rm (Outer integral):} {\it $\int_{\Omega_3} h^{p-1}\, \hat{w}_\mu \ud V_g = \beta_\mu \int_M h^{p-1}\, G_{x_0} \ud V_g + \mathrm{o}(\beta_\mu)$.}

\smallskip
\noindent{\it Proof.} In the outer region $M \setminus B_{2r_\mu}(x_0)$, one has $\hat{w}_\mu = \beta_\mu\, G_{x_0}$, which implies
\[
\int_{M \setminus B_{2r_\mu}} h^{p-1}\, \hat{w}_\mu \ud V_g = \beta_\mu \int_{M \setminus B_{2r_\mu}} h^{p-1}\, G_{x_0} \ud V_g = \beta_\mu \int_M h^{p-1}\, G_{x_0} \ud V_g - \beta_\mu \int_{B_{2r_\mu}} h^{p-1}\, G_{x_0} \ud V_g.
\]
Since $G_{x_0}(x) = \mathcal{O}(|y|^{2-n})$ near $x_0$ and $h^{p-1}$ is bounded, one has
\[
\int_{B_{2r_\mu}} h^{p-1}\, G_{x_0} \ud V_g = \mathcal{O}\!\left(\int_0^{2r_\mu} s^{2-n}\, s^{n-1} \ud s\right) = \mathcal{O}(r_\mu^2) = \mathrm{o}(1),
\]
by the integrability of $G_{x_0}$ near $x_0$. It follows that
\[
\int_{M \setminus B_{2r_\mu}} h^{p-1}\, \hat{w}_\mu \ud V_g = \beta_\mu \int_M h^{p-1}\, G_{x_0} \ud V_g + \mathrm{o}(\beta_\mu).
\]
This proves Claim~2.

\medskip
\noindent{\bf Claim~3} {\rm (Transition annulus):} {\it $\int_{\Omega_2} h^{p-1}\, \hat{w}_\mu \ud V_g = \mathrm{o}(\mu^{(n-2)/2})$.}

\smallskip
\noindent{\it Proof.} In the transition region $B_{2r_\mu} \setminus B_{r_\mu}$, the interpolation $\hat{w}_\mu = \eta_\mu U_\mu + (1 - \eta_\mu)\beta_\mu G_{x_0}$ satisfies $|\hat{w}_\mu| \lesssim \beta_\mu\, |y|^{2-n}$ uniformly. Since $h^{p-1}$ is bounded, one has
\[
\int_{B_{2r_\mu} \setminus B_{r_\mu}} h^{p-1}\, \hat{w}_\mu \ud V_g = \mathcal{O}\!\left(\beta_\mu \int_{r_\mu}^{2r_\mu} s^{2-n}\, s^{n-1} \ud s\right) = \mathcal{O}(\beta_\mu\, r_\mu^2) = \mathcal{O}(\mu^{(n-2)/2}\, \mu) = \mathrm{o}(\mu^{(n-2)/2}).
\]
This proves Claim~3.

\medskip
Combining Claims~1--3, since $\beta_\mu = \alpha_n(n-2)\omega_{n-1}\, \mu^{(n-2)/2}$ and 
$n/2 > (n-2)/2$ for all $n\in\mathbb{N}$,
the integrals over $\Omega_1$ and $\Omega_2$ are absorbed into $\mathrm{o}(\beta_\mu)$, and one obtains
\[
\int_M h^{p-1}\, w_\mu \ud V_g = \frac{\beta_\mu}{\|U_1\|_{L^p(\R^n)}}\, \int_M h^{p-1}\, G_{x_0} \ud V_g + \mathrm{o}(\mu^{(n-2)/2}).
\]
Taking the supremum over $h \in \mathscr{O}_1$, squaring, and dividing by $\mu^{n-2}$, the limit exists and yields
\begin{align*}
    \mathsf{m}_\infty(x_0) &= \frac{\beta_\mu^2}{\mu^{n-2}\,\|U_1\|_{L^p(\R^n)}^2}\, \sup_{h \in \mathscr{O}_1} \left(\int_M h^{p-1}\, G_{x_0} \ud V_g\right)^2 \\
    &= \frac{[\alpha_n(n-2)\omega_{n-1}]^2}{\|U_1\|_{L^p(\R^n)}^2}\, \sup_{h \in \mathscr{O}_1} \left(\int_M h^{p-1}\, G_{x_0} \ud V_g\right)^2,
\end{align*}
which is \eqref{eq:minfty-formula}.
\end{proof}

\begin{remark}
\label{rmk:moment-positive}
Since Yamabe optimizers satisfy $h > 0$ on $M$ by the maximum principle and $G_{x_0} > 0$ on $M \setminus \{x_0\}$ by positivity of $L_g$, the integral $\int_M h^{p-1}\, G_{x_0} \ud V_g > 0$ for every $h \in \mathscr{O}_1$, whence $\mathsf{m}_\infty(x_0) > 0$. In particular, the mass dominance condition \eqref{mass dominance} is never vacuous.
\end{remark}

\subsection{Completion of the proof}

With the energy expansion and the sharp formula for $\mathsf{m}_\infty$ in hand, we turn to the proof of Theorem~\ref{theorem Schoen}.

\begin{proof}[Proof of Theorem~\ref{theorem Schoen}]
We use the test functions $w_\mu$ constructed above. Setting
\[
E_\mu := \mathcal{E}[w_\mu] = Y_\circ + a_\mu \quad {\rm and} \quad m_\mu := \mathsf{m}(w_\mu),
\]
the same computation as in \eqref{eq:S-expansion} and \eqref{eq:S-expansion-final} implies
\begin{equation}\label{eq:S-expansion-schoen}
\mathcal{S}_\gamma(w_\mu) = c_\gamma^{\rm bub} + \frac{Y_M}{Y_\circ^2}\, a_\mu + \frac{\gamma\, Y_M(Y_\circ - Y_M)}{2\, Y_\circ^2}\, m_\mu + \mathrm{o}(|a_\mu|) + \mathrm{o}(m_\mu).
\end{equation}
By Proposition~\ref{prop:schoen-energy}, it holds $a_\mu = -\beta_n \mathfrak{m}_{x_0}\, \mu^{n-2} + \mathrm{o}(\mu^{n-2})$, and by definition, one has 
\[
m_\mu \leqslant (\mathsf{m}_\infty(x_0) + \mathrm{o}(1))\, \mu^{n-2}.
\]
Inserting these into \eqref{eq:S-expansion-schoen}, one has
\begin{align*}
\mathcal{S}_\gamma(w_\mu) &\leqslant c_\gamma^{\rm bub} + \frac{Y_M\, \mu^{n-2}}{Y_\circ^2} \left[-\beta_n\, \mathfrak{m}_{x_0} + \frac{\gamma}{2}\,(Y_\circ - Y_M)\, \mathsf{m}_\infty(x_0)\right] + \mathrm{o}(\mu^{n-2}).
\end{align*}
The mass dominance condition \eqref{mass dominance} guarantees that the bracket is strictly negative, whence $\mathcal{S}_\gamma(w_\mu) < c_\gamma^{\rm bub}$ for $0 < \mu \ll 1$. Since $c_\gamma \leqslant \mathcal{S}_\gamma(w_\mu)$, this completes the proof.
\end{proof}

\begin{remark}
\label{rmk:comparison-yamabe}
In contrast with the Aubin case treated in \S\ref{sec:proof-thm2}, where the energy correction at order $\mu^4$ dominates the $\mathsf{m}$-term at order $\mu^{n-2}$ for $n \geqslant 6$, the Schoen energy correction \eqref{eq:schoen-energy} is of the same order $\mu^{n-2}$ as the $\mathsf{m}$-term. The mass dominance condition \eqref{mass dominance} is thus a necessary hypothesis for this approach, and its validity depends on the geometry of $(M,g)$.
\end{remark}

\appendix

\section{Energy expansions for the test functions}\label{sec:appendix-energy}
We provide the proofs of the classical energy expansions used in \S\ref{sec:proof-thm2} and \S\ref{sec:proof-schoen}. The computations for the Yamabe quotient are due to Aubin \cite{Aub} and Schoen \cite{Sch1}.
Our contribution is the stability expansion \eqref{eq:S-expansion-final}, which relies on these energy estimates as inputs.

\begin{proof}[Proof of Proposition~\ref{prop:aubin-energy}]
We follow \cite{Aub} and \cite[Section~6]{LP}. By conformal covariance of $L_g$, one has $\mathcal{E}_g[\hat{u}_\mu] = \mathcal{E}_{\tilde{g}}[\chi U_\mu]$. By a theorem of G\"unther \cite{Gun}, one can choose conformal normal coordinates so that $\det(\tilde{g}) = 1$, whence $\ud V_{\tilde{g}} = \ud y$.

In conformal normal coordinates centered at $x_0$, the inverse metric admits the expansion
\[
\tilde{g}^{ij}(y) = \delta^{ij} + \tfrac{1}{3}\, R_{ikj\ell}(x_0)\, y^k y^\ell + \mathcal{O}(|y|^3).
\]
Since $\det(\tilde{g}) = 1$, the volume element satisfies $\ud V_{\tilde{g}} = \ud y$ and the conformal normal condition yields $R_{\tilde{g}}(x_0) = 0$, so that $R_{\tilde{g}}(y) = \langle \nabla R_{\tilde{g}}(x_0), y \rangle + \mathcal{O}(|y|^2)$. Substituting the metric expansion into $\mathcal{E}_{\tilde{g}}[\chi U_\mu]$, one obtains
\begin{equation}\label{eq:aubin-energy-decomp}
\mathcal{E}_{\tilde{g}}[\chi U_\mu] = \int_{\R^n} |\nabla(\chi U_\mu)|^2 \ud y + c_n \int_{\R^n} R_{\tilde{g}}(y)\,(\chi U_\mu)^2 \ud y + I_\mu + \mathcal{O}(\mu^6),
\end{equation}
where
\begin{equation}
    \label{I mu}
    I_\mu := \tfrac{1}{3}\, R_{ikj\ell}(x_0) \int_{\R^n} y^k y^\ell\,\partial_i U_\mu\,\partial_j U_\mu \ud y.
\end{equation}
Recall that $\chi \in \mathcal{C}^\infty_c(\R^n)$ satisfies $\chi \equiv 1$ on $B_R(x_0)$ and $\supp(\chi) \subset B_{2R}(x_0)$ for some fixed $R > 0$ smaller than the injectivity radius.

\medskip
\noindent{\bf Claim~1} {\rm (Energy Expansion):} {\it $\mathcal{E}_{\tilde{g}}[\chi U_\mu] = 
Y_\circ\,\|U_1\|_{L^p(\R^n)}^2 + \mathcal{R}_n(\mu) + I_\mu + \mathcal{O}(\mu^{n-2})$, where}
\begin{equation}\label{eq:scalar-curvature-term}
    \mathcal{R}_n(\mu) := c_n \int_{\R^n} R_{\tilde{g}}(y)\, U_\mu^2 \ud y.
\end{equation}

\smallskip
\noindent{\it Proof.} By conformal invariance, the Euclidean gradient energy satisfies
\[
\int_{\R^n} |\nabla U_\mu|^2 \ud y = Y_\circ\,\|U_1\|_{L^p(\R^n)}^2.
\]
Since $\chi \equiv 1$ on $B_R$ and $|\chi| \leqslant 1$, one has $\nabla(\chi U_\mu) = \nabla U_\mu$ on $B_R$ and $|\nabla(\chi U_\mu)| \lesssim U_\mu$ on the annulus $B_{2R} \setminus B_R$. Since $U_\mu(y) = \mathcal{O}(\mu^{(n-2)/2} R^{-(n-2)})$ for $|y| \geqslant R$, the gradient energy of $\chi U_\mu$ on the annulus is $\mathcal{O}(\mu^{n-2})$. Similarly, $\int_{B_{2R} \setminus B_R} (\chi U_\mu)^2 \ud y = \mathcal{O}(\mu^{n-2})$. Therefore, it holds
\[
\int_{\R^n} |\nabla(\chi U_\mu)|^2 \ud y = \int_{\R^n} |\nabla U_\mu|^2 \ud y + \mathcal{O}(\mu^{n-2}) = Y_\circ\,\|U_1\|_{L^p(\R^n)}^2 + \mathcal{O}(\mu^{n-2}).
\]
For the scalar curvature term, since ${\rm supp}((\chi U_\mu)^2 - U_\mu^2)\subset B_{2R} \setminus B_R$, one has
\[
(\chi U_\mu)^2 - U_\mu^2\leqslant U_\mu^2 = \mathcal{O}(\mu^{n-2}\,|y|^{-2(n-2)}).
\]
Since $R_{\tilde{g}}$ is bounded on $B_{2R}$, it follows that
\[
c_n \int_{\R^n} R_{\tilde{g}}(y)\,\bigl[(\chi U_\mu)^2 - U_\mu^2\bigr] \ud y = \mathcal{O}\!\left(\mu^{n-2} \int_R^{2R} s^{n-1-2(n-2)} \ud s\right) = \mathcal{O}(\mu^{n-2}).
\]
In addition, since $R_{\tilde{g}}(x_0) = 0$, the scalar curvature integral in \eqref{eq:scalar-curvature-term}
satisfies 
\[
c_n \int_{\R^n} R_{\tilde{g}}(y)\,(\chi U_\mu)^2 \ud y = \mathcal{R}_n(\mu) + \mathcal{O}(\mu^{n-2}),
\]
where the first-order term $\langle \nabla R_{\tilde{g}}(x_0), y \rangle$ integrates to zero by parity. The dominant contribution comes from the second-order Taylor remainder $R_{\tilde{g}}(y) = \tfrac{1}{2}\partial_{ab}R_{\tilde{g}}(x_0)\,y^a y^b + \mathcal{O}(|y|^3)$. We write $\mathcal{R}_n(\mu) = \mathcal{H}_n(\mu) + \mathcal{O}(\mu^6)$, where the Hessian piece is
\[
\mathcal{H}_n(\mu) := c_n\,\frac{\partial_{ab}R_{\tilde{g}}(x_0)}{2} \int_{\R^n} y^a y^b\, \alpha_n^2\,\mu^{n-2}(|y|^2 + \mu^2)^{-(n-2)} \ud y.
\]
Passing to polar coordinates $y = r\omega$ with $r = |y|$ and $\omega \in \Ss^{n-1}$, the radial and angular variables separate. One can use the isotropy of the angular integral,
\begin{equation}\label{eq:isotropy}
 \int_{\Ss^{n-1}} \omega^a\omega^b \ud\omega = \tfrac{|\Ss^{n-1}|}{n}\,\delta^{ab},   
\end{equation}
to get
\[
\int_{\R^n} y^a y^b\, U_\mu^2 \ud y
= \int_0^R r^{n+1}(r^2 + \mu^2)^{-(n-2)} \ud r \alpha_n^2\mu^{n-2}\,\frac{|\Ss^{n-1}|}{n}\,\delta^{ab} + \mathcal{O}(\mu^{n-2}).
\]
The substitution $r = \mu s$ gives
\[
\int_0^R r^{n+1}(r^2 + \mu^2)^{-(n-2)} \ud r = \mu^{n+2-2(n-2)}\,\mathcal{J}_n(\mu) = \mu^{6-n}\,\mathcal{J}_n(\mu),
\]
where
\[
\mathcal{J}_n(\mu) := \int_0^{R/\mu} s^{n+1}(1+s^2)^{-(n-2)} \ud s.
\]
Together with the prefactor $\alpha_n^2\,\mu^{n-2}$ from $U_\mu^2$, the $\mu$-power in $\mathcal{H}_n(\mu)$ is $\mu^{n-2}\,\mu^{6-n} = \mu^4$. By \cite[Lemma~6.4]{LP}, conformal normal coordinates satisfy $\Delta R_{\tilde{g}}(x_0) = -\tfrac{1}{6}|W_g(x_0)|^2$, so the isotropy identity \eqref{eq:isotropy} applied to $\partial_{ab}R_{\tilde{g}}(x_0)$ yields
\[
\frac{\partial_{ab}R_{\tilde{g}}(x_0)}{2}\,\frac{|\Ss^{n-1}|}{n}\,\delta^{ab} = \frac{|\Ss^{n-1}|}{2n}\,\Delta R_{\tilde{g}}(x_0) = -\frac{|\Ss^{n-1}|}{12n}\,|W_g(x_0)|^2,
\]
and therefore
\begin{equation}\label{eq:Hmu}
\mathcal{H}_n(\mu) = -\frac{c_n\,\alpha_n^2\,|\Ss^{n-1}|}{12n}\,|W_g(x_0)|^2\,\mu^4\,\mathcal{J}_n(\mu).
\end{equation}
For $s \gg 1$, one has 
\[
(1+s^2)^{-(n-2)} = s^{-2(n-2)}(1 + s^{-2})^{-(n-2)} \sim s^{-2(n-2)},
\]
and so
\[
s^{n+1}(1+s^2)^{-(n-2)} \sim s^{n+1-2(n-2)} = s^{5-n}.
\]
The substitution $t = s^2$ yields
\[
\int_0^{+\infty} s^{n+1}(1+s^2)^{-(n-2)} \ud s = \tfrac{1}{2}\, B\!\left(\tfrac{n+2}{2},\, \tfrac{n-6}{2}\right),
\]
which converges if and only if $\tfrac{n-6}{2} > 0$, that is, $n \geqslant 7$. In this case, one has 
\[
\mathcal{J}_n(\mu) \to \tfrac{1}{2}\,B\!\bigl(\tfrac{n+2}{2},\,\tfrac{n-6}{2}\bigr) \quad {\rm as} \quad \mu \to 0.
\]
In addition, since $\mathcal{R}_n(\mu) = \mathcal{H}_n(\mu) + \mathcal{O}(\mu^6)$, substituting into \eqref{eq:Hmu} yields
\[
\mathcal{R}_n(\mu) = -C_n''\,|W_g(x_0)|^2\,\mu^4 + \mathrm{o}(\mu^4) \quad {\rm with} \quad C_n'' := \frac{c_n\,\alpha_n^2\,|\Ss^{n-1}|}{24n}\,B\!\left(\tfrac{n+2}{2},\,\tfrac{n-6}{2}\right) > 0,
\]
which proves Claim~1 for $n \geqslant 7$.

For $n = 6$, the integrand has a logarithmic singularity at infinity. Indeed, writing
\[
s^7(1+s^2)^{-4} = s^7\, s^{-8}\,(1 + s^{-2})^{-4} = s^{-1}\,(1 + s^{-2})^{-4} = s^{-1}\,\bigl(1 - 4s^{-2} + \mathcal{O}(s^{-4})\bigr) = s^{-1} + \mathcal{O}(s^{-3}),
\]
we get the estimate
\[
\int_0^{R/\mu} s^7(1+s^2)^{-4} \ud s = \int_0^1 s^7(1+s^2)^{-4} \ud s + \int_1^{R/\mu} \bigl[s^{-1} + \mathcal{O}(s^{-3})\bigr] \ud s = \log(R/\mu) + \mathcal{O}(1),
\]
where we used 
\[
\int_0^1 s^7(1+s^2)^{-4} \ud s<\infty.
\]

For $n = 6$, substituting $\mathcal{J}_6(\mu) = \log(R/\mu) + \mathcal{O}(1)$ into \eqref{eq:Hmu} and using the decomposition $\log(R/\mu) = \log(1/\mu) + \log R$, one obtains
\[
\mathcal{R}_n(\mu) = -\frac{c_6\,\alpha_6^2\,|\Ss^5|}{72}\,|W_g(x_0)|^2\,\mu^4\log(1/\mu) + \mathcal{O}(\mu^4),
\]
which by setting $C_6'' := \tfrac{c_6\,\alpha_6^2\,|\Ss^5|}{72}$ proves Claim~1 for $n = 6$.

\medskip
\noindent{\bf Claim~2} {\rm (Vanishing):}  {\it $I_\mu = 0$.}

\smallskip
\noindent{\it Proof.} We decompose the Riemann tensor at $x_0$ as
\[
R_{ikj\ell} = W_{ikj\ell} + \frac{1}{n-2}\left(R_{ij}\delta_{k\ell} - R_{i\ell}\delta_{kj} + R_{k\ell}\delta_{ij} - R_{kj}\delta_{i\ell}\right) - \frac{R}{(n-1)(n-2)}\left(\delta_{ij}\delta_{k\ell} - \delta_{i\ell}\delta_{kj}\right),
\]
where all quantities are evaluated at $x_0$ and $\tilde{g}_{ij}(x_0) = \delta_{ij}$. Since $U_\mu$ is radially symmetric, one has 
\[
\partial_i U_\mu = U_\mu'(|y|)\, \tfrac{y^i}{|y|} \quad {\rm and} \quad 
\partial_i U_\mu\,\partial_j U_\mu = (U_\mu')^2\, \frac{y^i y^j}{|y|^2}.
\]
Substituting into \eqref{I mu} and passing to polar coordinates $y = s\omega$ with $\omega \in \Ss^{n-1}$, the radial and angular variables separate:
\[
I_\mu = \frac{1}{3}\, R_{ikj\ell}(x_0) \int_0^\infty (U_\mu')^2\, s^{n+1} \ud s \int_{\Ss^{n-1}} \omega^i \omega^j \omega^k \omega^\ell \ud \omega:=\frac{1}{3}\, R_{ikj\ell}(x_0) \, C(n,\mu) \int_{\Ss^{n-1}} \omega^i \omega^j \omega^k \omega^\ell \ud \omega
\]
with $C_{\mu}>0$.
The isotropic tensor identity
\begin{equation}\label{eq:angular-average}
\int_{\Ss^{n-1}} \omega^i \omega^j \omega^k \omega^\ell \ud \omega = \frac{|\Ss^{n-1}|}{n(n+2)}\,\left(\delta^{ij}\delta^{k\ell} + \delta^{ik}\delta^{j\ell} + \delta^{i\ell}\delta^{jk}\right)
\end{equation}
then yields
\[
I_\mu = \frac{C(n,\mu)\,|\Ss^{n-1}|}{3\,n(n+2)}\, R_{ikj\ell}(x_0)\,\left(\delta^{ij}\delta^{k\ell} + \delta^{ik}\delta^{j\ell} + \delta^{i\ell}\delta^{jk}\right).
\]

We now show that $R_{ikj\ell}(x_0)\,T^{ikj\ell} = 0$ by checking each piece of the Kulkarni--Nomizu decomposition.

\smallskip
\noindent{\it Step 1 (Ricci and scalar curvature terms):} {\it $(R_{ikj\ell} - W_{ikj\ell})\,T^{ikj\ell} = 0$.}

\smallskip
\noindent Indeed, write $T^{ikj\ell} := \delta^{ij}\delta^{k\ell} + \delta^{ik}\delta^{j\ell} + \delta^{i\ell}\delta^{jk}$ for the isotropic tensor in \eqref{eq:angular-average}. Using $R_{ab}\,\delta^{ab} = R$ and $\delta_{ab}\,\delta^{ab} = n$, the three summands of $T^{ikj\ell}$ implies 
\[
R_{ij}\,\delta_{k\ell}\, T^{ikj\ell}
= R_{ij}\,\delta_{k\ell}\,\delta^{ij}\delta^{k\ell}
+ R_{ij}\,\delta_{k\ell}\,\delta^{ik}\delta^{j\ell}
+ R_{ij}\,\delta_{k\ell}\,\delta^{i\ell}\delta^{jk}
= Rn + R + R = (n+2)\,R.
\]
By the same index contractions, it follows
\[
R_{i\ell}\,\delta_{kj}\, T^{ikj\ell} = R_{k\ell}\,\delta_{ij}\, T^{ikj\ell} = R_{kj}\,\delta_{i\ell}\, T^{ikj\ell} = (n+2)\,R.
\]
Since the Ricci piece carries alternating signs, it holds
\[
\frac{1}{n-2}\bigl[(n+2)R - (n+2)R + (n+2)R - (n+2)R\bigr] = 0.
\]
Similarly, one has
\[
\delta_{ij}\delta_{k\ell}\, T^{ikj\ell} = n^2 + 2n = \delta_{i\ell}\delta_{kj}\, T^{ikj\ell},
\]
so $(\delta_{ij}\delta_{k\ell} - \delta_{i\ell}\delta_{kj})\, T^{ikj\ell} = 0$.
In particular, the Ricci and scalar curvature pieces of the Kulkarni--Nomizu decomposition both contract to zero against $T^{ikj\ell}$, so
\[
I_\mu = \tfrac{1}{3}\,W_{ikj\ell}(x_0)\, T^{ikj\ell}.
\]
This reduction is purely algebraic and does not require $R_{\tilde{g}}(x_0) = 0$.

\smallskip
\noindent{\it Step 2 (Weyl term):} {\it $W_{ikj\ell}\, T^{ikj\ell} = 0$.}

\smallskip
\noindent In fact, we check the three contractions individually.

\smallskip
\noindent\textrm{(i)} For $W_{ikj\ell}\,\delta^{ik}\delta^{j\ell}$, setting $k = i$ and $\ell = j$ yields $\sum_{i,j} W_{iijj}$. By the antisymmetric property $W_{abcd} = -W_{bacd}$, each summand satisfies 
\[
W_{iijj} = -W_{iijj} = 0.
\]

\smallskip
\noindent\textrm{(ii)} For $W_{ikj\ell}\,\delta^{ij}\delta^{k\ell}$, setting $j = i$ and $\ell = k$ produces $\sum_{i,k} W_{ikik}$. The trace-free property of the Weyl tensor reads $g^{ac}W_{abcd} = 0$ for all $b, d$. At $x_0\in M$, where $\tilde{g}_{ij} = \delta_{ij}$, this implies $\sum_a W_{akak} = 0$ for each $k$, hence 
\[
\sum_{i,k} W_{ikik} = 0.
\]

\smallskip
\noindent\textrm{(iii)} For $W_{ikj\ell}\,\delta^{i\ell}\delta^{jk}$, setting $\ell = i$ and $k = j$ yields $\sum_{i,j} W_{ijji}$. Again, using the antisymmetry $W_{abcd} = -W_{abdc}$ in the second pair, $W_{ijji} = -W_{ijij}$, so that
\[
\sum_{i,j} W_{ijji} = -\sum_{i,j} W_{ijij} = 0,
\]
where the last equality follows from~(ii) after relabeling $j \mapsto k$.

Since all three contractions vanish, one has $I_\mu = 0$. Moreover, since $\det(\tilde{g}) = 1$, the $L^p$ normalization has no volume correction, {\it i.e.}
\[
\|\chi U_\mu\|_{L^p}^p = \int_{\R^n} |\chi U_\mu|^p \ud y = \|U_1\|_{L^p(\R^n)}^p + \mathcal{O}(\mu^n),
\]
where the error arises only from the cutoff tail. This proves Claim~2.

\medskip
Since $I_\mu = 0$, the energy decomposition \eqref{eq:aubin-energy-decomp} and the gradient estimate from the proof of Claim~1 yields
\[
\mathcal{E}_{\tilde{g}}[\chi U_\mu] = Y_\circ\,\|U_1\|_{L^p(\R^n)}^2 + \mathcal{R}_n(\mu) + \mathcal{O}(\mu^{n-2}).
\]
By Claim~1, one has
\[
\mathcal{R}_n(\mu)=
\begin{cases}
  -C_n''\,|W_g(x_0)|^2\,\mu^4 + \mathrm{o}(\mu^4), &{\rm if} \; n \geqslant 7\\
  -C_6''\,|W_g(x_0)|^2\,\mu^4\log(1/\mu) + \mathcal{O}(\mu^4), &{\rm if} \; n = 6.
\end{cases}
\]
Finally, dividing by
\[
\|\chi U_\mu\|_{L^p}^2 = \|U_1\|_{L^p(\R^n)}^2 + \mathrm{o}(1)
\]
and setting $C_n > 0$ accordingly, one obtains $a_\mu$ as stated.
\end{proof}

\begin{proof}[Proof of Proposition~\ref{prop:schoen-energy}]
As before, we decompose $M$ into the core $B_{r_\mu}(x_0)$, the transition annulus $A_\mu(x_0):=B_{2r_\mu}(x_0) \setminus B_{r_\mu}(x_0)$, and the outer region $M \setminus B_{2r_\mu}(x_0)$, and compute the integral \eqref{eq:energy-weak} on each. The energy satisfies the weak formulation
\begin{equation}\label{eq:energy-weak}
\mathcal{E}[\hat{w}_\mu] = \int_M \hat{w}_\mu\, L_g \hat{w}_\mu \ud V_g,
\end{equation}
where $L_g = -\Delta_g + c_n R_g$ is the conformal Laplacian.

\medskip
\noindent{\bf Claim~1} {\rm (Bubble core):} {\it $\int_{B_{r_\mu}(x_0)} \hat{w}_\mu\, L_g \hat{w}_\mu \ud V_g = Y_\circ\,\|U_1\|_{L^p(\R^n)}^2 + \mathrm{o}(\mu^{n-2})$.}

\smallskip
\noindent{\it Proof.} In this region $\hat{w}_\mu = U_\mu$. In conformal normal coordinates, the metric is flat to leading order, whence
\[
\int_{B_{r_\mu}} U_\mu\, L_g U_\mu \ud V_g = \int_{B_{r_\mu}} |\nabla U_\mu|^2 \ud y + \mathrm{o}(\mu^{n-2}) = \int_{\R^n} |\nabla U_\mu|^2 \ud y + \mathrm{o}(\mu^{n-2}),
\]
where we used that the tail satisfies
\[
\int_{\R^n \setminus B_{r_\mu}} |\nabla U_\mu|^2 \ud y = \mathcal{O}(\mu^{n-2} r_\mu^{4-n}) = \mathrm{o}(\mu^{n-2})
\quad {\rm for} \quad r_\mu = \mu^{1/2}.\]
In addition, by conformal invariance, the Euclidean energy of the rescaled bubbles is
\[
\int_{\R^n}|\nabla U_\mu|^2 \ud x=Y_\circ\,\|U_1\|_{L^p(\R^n)}^2,
\]
which proves Claim~1.

\medskip
\noindent{\bf Claim~2} {\rm (Outer region):} {\it $\int_{M \setminus B_{2r_\mu}(x_0)} \hat{w}_\mu\, L_g \hat{w}_\mu \ud V_g = 0$.}

\smallskip
\noindent{\it Proof.} Here $\hat{w}_\mu = \beta_\mu\, G_{x_0}$. Since $L_g G_{x_0} = \delta_{x_0}$ and $x_0 \notin M \setminus B_{2r_\mu}(x_0)$, the weak formulation \eqref{eq:energy-weak} yields
\[
\int_{M \setminus B_{2r_\mu}} \hat{w}_\mu\, L_g \hat{w}_\mu \ud V_g = \beta_\mu^2 \int_{M \setminus B_{2r_\mu}} G_{x_0}\, L_g G_{x_0} \ud V_g = 0,
\]
since $L_g G_{x_0} = 0$ on $M \setminus \{x_0\}$, and Claim~2 is proved.

\medskip
\noindent{\bf Claim~3} {\rm (Transition annulus):} {\it The transition annulus contributes $-\beta_n\, \|U_1\|_{L^p(\R^n)}^2\, \mathfrak{m}_{x_0}\, \mu^{n-2} + \mathrm{o}(\mu^{n-2})$, with $\beta_n$ as in \eqref{eq:beta-def}.}

\smallskip
\noindent{\it Proof.} The matching constant $\beta_\mu = \alpha_n(n-2)\omega_{n-1}\mu^{(n-2)/2}$ ensures that $U_\mu$ and $\beta_\mu G_{x_0}$ agree to leading order at $|y|^{2-n}$, so that
\[
\hat{w}_\mu(x) - U_\mu(y) = \beta_\mu\, \mathfrak{m}_{x_0} + \mathcal{O}(|y|^{3-n}\mu^{(n-2)/2})
\]
in the transition region. The mass $\mathfrak{m}_{x_0}$ enters via integration by parts: the mismatch between $\hat{w}_\mu$ and $U_\mu$ at the inner boundary $\partial B_{r_\mu}$ produces a surface integral proportional to $\mathfrak{m}_{x_0}$, while the cutoff and curvature contributions are of lower order. The precise computation of the coefficient $\beta_n$ follows from the integration-by-parts argument of Schoen \cite{Sch1} (see also \cite{Schoen1989} and \cite[Proposition~7.1]{LP}).

\medskip
At last, by combining Claims~1--3, one obtains
\[
\mathcal{E}[\hat{w}_\mu] = Y_\circ\, \|U_1\|_{L^p(\R^n)}^2 - \beta_n\, \|U_1\|_{L^p(\R^n)}^2\, \mathfrak{m}_{x_0}\, \mu^{n-2} + \mathrm{o}(\mu^{n-2}).
\]
Dividing by $\|\hat{w}_\mu\|_{L^p(M)}^2 = \|U_1\|_{L^p(\R^n)}^2 + \mathrm{o}_\mu(1)$, one obtains \eqref{eq:schoen-energy} with $\beta_n$ as in \eqref{eq:beta-def}.
\end{proof}

\section*{Acknowledgments}
J.\,H.\,A.\ is partially supported by FAPESP grants \#2020/07566-3, \#2021/15139-0, and \#2023/15567-8, and CNPq grants \#409764/2023-0, \#443594/2023-6, \#441922/2023-6, and \#306014/2025-4.
T.\,K.\ is funded by the Deutsche Forschungsgemeinschaft (DFG, German Research Foundation), project numbers 555837013 and 561401741.
J.\,R.\ is funded by the Deutsche Forschungsgemeinschaft (DFG, German Research Foundation), project number 561401741.
J.\,W.\ is partially supported by HKGRF grant \#14309824.

\section*{Notation}

We collect the main notation used throughout the paper.
\begin{itemize}[leftmargin=2em, itemsep=2pt]
\item[--] $(M, g)$ is a closed $n$-dimensional Riemannian manifold with $n \geqslant 3$;
\item[--] $R_g$, $\Delta_g$, $\ud V_g$ are the scalar curvature, Laplace--Beltrami operator, and volume form;
\item[--] $H^1(M)$ is the Sobolev space on $M$ with standard norm $\|u\|_{H^1(M)}^2 := \int_M (|\nabla u|^2 + |u|^2) \ud V_g$;
\item[--] $p = 2^*:= \tfrac{2n}{n-2}$ is the critical Sobolev exponent of $H^1(M)$;
\item[--] $\gamma \geqslant 2$ is the stability exponent;
\item[--] $\omega_{n-1} = |\Ss^{n-1}|$ is the volume of the unit $(n-1)$-sphere;
\item[--] $c_n = \tfrac{n-2}{4(n-1)}$ is the conformal normalizing constant;
\item[--] $L_g = -\Delta_g + c_n R_g$ is the conformal Laplacian;
\item[--] $\mathcal{E}[u] = \int_M(|\nabla u|^2 + c_n R_g u^2)\ud V_g$ is the conformal energy;
\item[--] $\mathcal{Q}(u) = \mathcal{E}[u]/\|u\|_{L^p(M)}^2$ is the Yamabe quotient;
\item[--] $Y_M = Y(M,[g])$ is the Yamabe constant of $(M,[g])$;
\item[--] $g_\circ$ is the round metric on $\Ss^n$;
\item[--] $Y_\circ = Y(\Ss^n,[g_\circ]) = \tfrac{n(n-2)}{4}\,\omega_n^{2/n}$ is the Yamabe constant of the round 
metric, where $\omega_n = |\Ss^n|$;
\item[--] $U_1(y) = \alpha_n(1/(1+|y|^2))^{(n-2)/2}$ is the standard Aubin--Talenti bubble;
\item[--] $U_\mu(y) = \alpha_n(\mu/(\mu^2+|y|^2))^{(n-2)/2}$ is the scaled Aubin--Talenti bubble in \eqref{eq:talenti-aubin};
\item[--] $\alpha_n = [n(n-2)]^{(n-2)/4}$ is a normalizing dimensional constant;
\item[--] $\mathcal{S}_\gamma: H^1(M)\setminus \mathscr{O}\to\R$ is the stability quotient in \eqref{defn_stability};
\item[--] $\mathscr{O} = \{ u \in \mathcal{C}^\infty(M) : \mathcal{Q}(u) = Y_M,\, u \geqslant 0 \}$ is the Yamabe optimizer manifold in \eqref{eq:optimizer-set};
\item[--]  $\mathscr{O}_1 = \{ h \in \mathscr{O} : \|h\|_{L^p} = 1 \}$ is the set of unit $L^p$-normalized optimizers in \eqref{O1 compact};
\item[--] $c_\gamma(M,g)$ is the best stability constant in \eqref{eq:ENS-def};
\item[--] $c_\gamma^{\loc}(M,g)$ is the local stability constant in \eqref{eq:stability_constant_loc};
\item[--] $c_\gamma^{\rm bub}(M,g)$ is the bubble threshold in \eqref{eq:stability_constant_bubble};
\item[--] $G_{x_0}$ is the Green function of $L_g$ with pole at $x_0 \in M$ in \eqref{Green function expansion};
\item[--] $\mathfrak{m}_{x_0}$ is the mass of $G_{x_0}$ in \eqref{Green function expansion};
\item[--] $\mathsf{m}: H^1(M)\to [0,\infty)$ is the stability mass in \eqref{eq:momentum} (see Remark~\ref{rmk:stability-mass});
\item[--] $\mathsf{m}_\infty(x_0)>0$ is the rescaled stability mass coefficient in \eqref{eq:minfty-def};
\item[--] $\beta_n > 0$ is the dimensional constant in \eqref{eq:beta-def};
\item[--] $\tilde{g} = \Lambda^{4/(n-2)} g$ is the conformal normal metric centered at $x_0$;
\item[--] $\chi \in \mathcal{C}^\infty_c(\R^n)$ is a smooth cutoff function with $\chi \equiv 1$ near the origin;
\item[--] $u_\mu = \hat{u}_\mu / \|\hat{u}_\mu\|_{L^p(M)}$ is the normalized Aubin test function in \eqref{eq:bubbles_family};
\item[--] $w_\mu = \hat{w}_\mu / \|\hat{w}_\mu\|_{L^p(M)}$ is the normalized Schoen test function in \eqref{eq:schoen_family};
\item[--] $a_\mu$ is the energy correction $\mathcal{E}[u_\mu] - Y_\circ$ or $\mathcal{E}[w_\mu] - Y_\circ$;
\item[--] $m_\mu = \mathsf{m}(u_\mu)$ or $\mathsf{m}(w_\mu)$ is the stability mass of the test function;
\item[--] $\dist(u, \mathscr{O}) = \inf_{h \in \mathscr{O}} \mathcal{E}[u-h]^{1/2}$ is the $H^1$-distance to the optimizer manifold in \eqref{dist definition};
\item[--] $\mathrm{o}(\cdot)$, $\mathcal{O}(\cdot)$ is the standard asymptotic notation;
\item[--] $\mathrm{o}_k(1)$ denotes any quantity that tends to zero as $k \to \infty$, and similarly $\mathrm{o}_\mu(\cdot)$ as $\mu \to 0$;
\item[--] $f \lesssim g$ means $f \leqslant C\, g$ for a constant $C > 0$ depending only on $n$ and $(M,g)$; similarly $f \gtrsim g$ means $f \geqslant C\, g$;
\item[--] $C_*$ denotes a positive constant whose subscript indicates its dependence, {\it e.g.}, $C_\eps$ depends on $\eps$ (and on fixed data $n, g$).
\end{itemize}
In the proofs in \S\ref{sec:proof-thm1}--\S\ref{sec:proof-schoen}, the underlying manifold $(M,g)$ and exponent $\gamma \geqslant 2$ are fixed and we use the shorthand $c_\gamma$, $c_\gamma^{\loc}$, $c_\gamma^{\rm bub}$ throughout.

\bibliography{references.bib}
\bibliographystyle{abbrv}

\end{document}